\documentclass[onefignum,onetabnum]{siamart190516}
\usepackage[T1]{fontenc}
\usepackage[english]{babel}
\usepackage{amsmath}
\usepackage{amssymb}
\usepackage{amsfonts}
\usepackage{array}
\usepackage{graphicx}
\usepackage{epsfig}
\usepackage{float}
\usepackage{color}
\usepackage{enumitem}  
\usepackage{epstopdf}
\usepackage{stmaryrd}
\usepackage{standalone}
\usepackage{tikz, subfigure}

\usepackage{todonotes}

\usetikzlibrary{shapes, patterns, arrows}
\usetikzlibrary{calc}
\usepackage{pgfplots}
\usepgfplotslibrary{fillbetween}
\pgfplotsset{compat=newest}
\usepackage{pgfplotstable}
\usepgfplotslibrary{groupplots}
\usepackage{diagbox}

\usepackage{booktabs}  
\usepackage{capt-of}
\usepackage{multirow}
\usepackage[mathlines]{lineno}

\crefname{section}{Section}{Sections}
\crefname{subsection}{Section}{Sections}
\crefname{subsubsection}{Section}{Sections}
\crefname{example}{Example}{Examples}

\newcommand{\jump}[1]{\ensuremath{[\![#1]\!]} }
\newcommand{\avg}[1]{\ensuremath{\left\{\!\!\left\{#1\right\}\!\!\right\}} }
\newcommand{\enorm}[1]{{\left\vert\kern-0.25ex\left\vert\kern-0.25ex\left\vert #1 
    \right\vert\kern-0.25ex\right\vert\kern-0.25ex\right\vert}}

\newcommand{\cell}{{K}}
\newcommand{\face}{{F}}
\renewcommand{\div}{\nabla\!\cdot}
\newcommand{\grad}{\nabla}

\renewcommand{\vec}[1]{\boldsymbol{#1}}
\newcommand{\meas}[1]{\lvert #1 \rvert}

\newcommand{\vels}{{\vec{u}_S}}
\newcommand{\velstest}{{\vec{v}_S}}
\newcommand{\velstilde}{{\vec{\tilde{u}}_S}}
\newcommand{\fullperm}{{\vec{K}}}
\newcommand{\ident}{{\vec{I}}}
\newcommand{\perm}{k}
\newcommand{\kovermu}{{\kappa}}

\newcommand{\betatangentfull}{{\frac{\mu\alpha}{\sqrt{\perm}}}}
\newcommand{\betatangent}{{\beta_\tau}}

\newcommand{\betanormalfull}{{\kappa^{-1}h_K}}
\newcommand{\betanormal}{\beta_n} 

\newcommand{\stress}{\vec{\sigma}}
\newcommand{\stressvar}{{\vec{\sigma}(\vels, p_S)}}
\newcommand{\symgradu}{\vec{\epsilon}(\vels)}
\newcommand{\symgrad}[1]{\vec{\epsilon}(#1)}
\newcommand{\symgradop}{{\vec{\epsilon}}}

\newcommand{\uonormal}{{\vec{n}}}
\newcommand{\utangent}{{\vec{\tau}}}
\newcommand{\uonormals}{{\vec{n}_S}}
\newcommand{\uonormald}{{\vec{n}_D}}

\newcommand{\bfS}{{\vec{f}_S}}
\newcommand{\bV}{{\vec{V}}}
\newcommand{\bH}{{\vec{H}}}
\newcommand{\bu}{{\vec{u}}}
\newcommand{\bv}{{\vec{v}}}

\newcommand\Pnot{\ensuremath{{P}_0}}

\newcommand{\tracenormal}{{T^{}_{n}}}
\newcommand{\tracenormalprime}{{T^{\prime}_{n}}}
\newcommand{\tracetangent}{{T^{}_{\tau}}}
\newcommand{\tracetangentprime}{{T^{\prime}_{\tau}}}

\title{
  Robust 
  monolithic solvers for
  the Stokes-Darcy problem 
  with the Darcy equation in primal form
  \thanks{Submitted to the editors \today. The authors are listed in alphabetical order.\funding{
WMB acknowledges support from the Dahlquist Research Fellowship, funded by Comsol AB.      
The work of TK was financially supported by the European Union's Horizon 2020 research and innovation programme under the Marie Sk{\l{}}odowska-Curie grant agreement No~801133.      
MK acknowledges support from the Research Council of Norway (NFR) grant No~303362.
KAM acknowledges support from the Research Council of Norway grant No~300305 and 301013.}}
}

\author{
  Wietse M. Boon\thanks{KTH Royal Institute of Technology, Stockholm, Sweden (\email{wietse@kth.se})} \and  
  Timo Koch\thanks{Department of Mathematics, University of Oslo, Norway (\email{timokoch@uio.no, kent-and@uio.no})} \and%
  Miroslav Kuchta\thanks{Simula Research Laboratory, Oslo, Norway (\email{miroslav@simula.no}). Corresponding author} \and%
  Kent-Andr{\'e} Mardal\footnotemark[3] \footnotemark[4]%
}

\begin{document}

\maketitle

\begin{abstract}
  We construct mesh-independent and parameter-robust monolithic solvers
  for the coupled primal Stokes-Darcy problem. Three different formulations and
  their discretizations in terms of conforming and
  non-conforming finite element methods and finite volume methods are considered.
  In each case, robust preconditioners are derived using a unified theoretical
  framework. In particular, the suggested preconditioners utilize operators in fractional Sobolev spaces.
  Numerical experiments demonstrate the parameter-robustness of the proposed solvers.
\end{abstract}



 
\newtheorem{thm}{Theorem}[section]
\newtheorem{prop}{Property}[section]
\newtheorem{remark}{Remark}[section]
\newtheorem{example}{Example}[section]

\section{Introduction}\label{sec:intro}

In this work, we propose efficient solvers for multi-physics systems where
a moving fluid (e.g. channel flow) governed by the Stokes equations in one sub-domain 
interacts with fluid flow in porous media described by the Darcy equation in a
neighboring sub-domain. The main contribution is a framework which allows us
to construct parameter-robust preconditioners for iterative solvers of linear
systems arising from different discretizations of the coupled Stokes-Darcy problem. The theory is
confirmed and complemented by extensive numerical experiments building on modern and open-source numerical software frameworks.

Systems exhibiting free flow coupled with porous medium flow are ubiquitous in nature
appearing in numerous environmental, industrial (see e.g. \cite{discacciati2009navier} and references therein), and medical applications~\cite{ROHAN2021106404} .
Discretization of the Stokes-Darcy problem is challenging with
many finite element (e.g. \cite{discacciati2002mathematical, girault2014mortar, layton2002coupling, karper2009unified, riviere_dg, riviere2005locally, gatica2008conforming, burman2007unified, badia2009unified}) and finite volume schemes
(e.g. \cite{Shiue2018,Schneider2020}) devised with the aim to obtain robust approximation properties.
Moreover,
the coupled system presents a difficulty for construction of numerical solvers
as in the applications the problem parameters weighting different terms of the equations
may differ by several orders of magnitude due to, for example, variations in material parameters or
large contrast of length scales (e.g. micro/macro-circulation modelling \cite{koch2020multiscale, smith2007interstitial}).

These challenges have been addressed in a number of works.
In general, we can distinguish between monolithic approaches
(where all the problem unknowns are solved for at once) and domain-decomposition (DD)
techniques (where the coupled system is solved using iterations between
the sub-domain problems). In the context of primal Stokes-Darcy problem,
which will be studied in this work, DD solvers have been established e.g. in \cite{discacciati2002mathematical, discacciati2007robin, discacciati2018optimized, chen2011parallel, caiazzo2014classical}.
Monolithic solvers have been developed primarily for the non-symmetric problem formulation in terms of
Krylov solvers (GMRes) with block-diagonal and triangular preconditioners
\cite{cai2009preconditioning} or constrained indefinite preconditioners \cite{chidyagwai2016constraint}. 
However, existing solvers are typically robust only in certain parameter
regimes (cf. \cite{cai2009preconditioning, chidyagwai2016constraint}) 
or rely on algorithmic parameters that may be difficult to tune (e.g. Robin parameters in DD \cite{discacciati2018optimized}).

Monolithic methods are in particular popular in applications with more complex physics,
e.g. \cite{Mosthaf2011,Baber2012,Coltman2020,Ackermann2021}, for their property that the interface conditions are fulfilled up
to numerical precision independent of tuning parameters, and the practical
observation that monolithic schemes often outperform DD schemes in cases
where the DD solver requires many sub-domain iterations. This can also be the case if optimal DD parameters are unknown for the specific problem and parameters or costly to determine. For completeness, we mention that there are also works that successfully apply DD techniques for problems with more complex physics, e.g. \cite{Birgle2018}.
%

In \cite{wietse,holter2020robust, Luo2017UzawaSI},
robust solvers for the Stokes-Darcy problem with Darcy equation in mixed form (see e.g.~\cite{layton2002coupling, galvis2007non}) are constructed.
While the mixed form has the advantage in the finite element context of ensuring local mass conservation,
the total number of degrees of freedom is significantly reduced with the Darcy problem in the primal form.
Finite volume schemes feature local mass conservation by construction in both cases.

In the following, we construct robust monolithic solvers for the primal
Stokes-Darcy system. More precisely, by considering different discretizations 
of the coupling conditions, we derive three different symmetric formulations
which are amenable to discretization by finite element (FEM) or finite volume methods (FVM).
Well-posedness of the formulations is established
within an abstract framework and consequently block-diagonal preconditioners are
constructed by operator preconditioning \cite{Mardal2011}. A crucial component of the
analysis is the formulation in terms of fractional norms on the interface between the
sub-domains. In turn, the proposed preconditioners utilize non-standard and non-local operators. However, as the number of degrees of freedom on the interface is often small, we
demonstrate that the preconditioners are feasible also in practical
applications.

Our work is structured as follows. In \cref{sec:problem}, we state the governing equations and coupling conditions, introduce the three variational formulations considered in this work, and show in a motivating example that a simple idea based on standard norms does not lead to a parameter-robust preconditioner.
An abstract theory is then developed in \cref{sec:abstract} and applied to the different formulations. Numerical experiments showcasing 
robustness of the proposed preconditioners and their efficiency are presented and discussed
in \cref{sec:numeric}.


\section{Problem formulation}\label{sec:problem}

Let $\Omega_S, {\Omega}_D\subset \mathbb{R}^d$, $d \in \{ 2, 3\}$, be two nonoverlapping Lipschitz
domains sharing a common interface $\Gamma = \partial{\Omega}_S \cap \partial{\Omega}_D \subset \mathbb{R}^{d-1}$.
Let $\Omega_D$ represent a porous medium in which we consider Darcy flow in primal form, i.e. formulated solely
in terms of pressure $p_D$,
\begin{equation}\label{eq:darcy}
\div \left(- \mu^{-1} \fullperm \grad {p}_D \right) = {f}_D,
\end{equation}
with constant fluid viscosity $\mu > 0$, and isotropic and homogeneous intrinsic permeability
$\fullperm = \perm\vec{\ident}$. For notational convenience, we further
let $\kovermu := \mu^{-1} \perm$.

In the free-flow domain $\Omega_S$, we consider the Stokes problem,
\begin{subequations}\label{eq:stokes}
\begin{align}
  -\div \stressvar &= \bfS, \label{eq:stokes_mom} \\
  -\div \vels &= 0, \label{eq:stokes_mass}
\end{align}
\end{subequations}
with $\stressvar = 2 \mu \symgradu - p_S\ident$ and $\symgradu = \frac{1}{2} \left( \grad\vels + \grad\vels^T \right)$.

To couple the Stokes and Darcy systems, let $\uonormal := \uonormals$ be
the outer normal of the Stokes domain and let $\utangent := \ident - (\uonormal \otimes \uonormal)$ be the projection onto the tangent bundle of the interface.
The following conditions are then assumed to hold on the interface $\Gamma$
\begin{subequations}\label{eq:interface}
\begin{align}
  \utangent \cdot \stressvar\cdot \uonormal + \betatangent \utangent \cdot \vels&= \vec{0},
  \label{eq:BJS}  \\
  \uonormal \cdot\stressvar\cdot \uonormal + {p}_D &= 0,   \label{eq:stress}  \\
  \uonormal\cdot \vels + \uonormal \cdot \kovermu \nabla p_D &= 0.   \label{eq:mass_csrv}  
\end{align}
\end{subequations}
Here, the first of the coupling conditions is the well-established Beavers-Joseph-Saffman (BJS)
condition \cite{Beavers1967a,Saffman1971a,Mikelic2000a} with $\betatangent := \betatangentfull$,
and constant $\alpha \geq 0$. Finally, conditions \eqref{eq:stress}-\eqref{eq:mass_csrv} enforce normal stress continuity
and mass conservation.

To close the coupled problem \eqref{eq:darcy}-\eqref{eq:interface},
we prescribe the following (homogeneous) boundary conditions
\begin{subequations} \label{eq: bcs}
\begin{align}
  \vels &= 0 \quad \text{on } \Gamma_S^{\vec{u}}, &
  \uonormals \cdot \stressvar &= 0 \quad \text{on } \Gamma_S^{\stress} \ne \emptyset, \\
  - \uonormald \cdot \kappa \nabla p_D &= 0 \quad \text{on } \Gamma_D^u, &
  p_D &= 0 \quad \text{on } \Gamma_D^p \ne \emptyset.
\end{align}
\end{subequations}

Here, we assume that $\Gamma_S^u \cup \Gamma_S^{\stress} \cup \Gamma$ forms a
disjoint decomposition of $\partial \Omega_S$ and, analogously, $\Gamma_D^{\vec{u}} \cup \Gamma_D^p \cup \Gamma$
is a disjoint partition of $\partial \Omega_D$. Since we assume that both $\Gamma_S^{\stress}$ and
$\Gamma_D^p$ have positive measure, $\Gamma$ cannot be a closed surface (or curve in 2D).
In turn, we make the assumption that its boundary touches the boundary sections on which Stokes stress and
Darcy flux boundary conditions are imposed, i.e. $\partial \Gamma \subseteq \partial \Gamma_S^{\stress} \cup \partial \Gamma_D^u$.
These assumptions are made specifically to simplify the analysis in \cref{sec:abstract} and will be relaxed in the numerical experiments of \cref{sec:numeric}.


\subsection{Three variational formulations} \label{sub: three variational formulations}
In this work, we focus on three different formulations of the coupled problem \eqref{eq:darcy}-\eqref{eq: bcs}. The formulations differ in the manner in which the flux continuity condition \eqref{eq:mass_csrv} is incorporated. The first uses the trace of $p_D$ on the interface to enforce this condition and we call this formulation the \emph{Trace} (Tr) formulation. The second formulation uses the interface pressure as a Lagrange multiplier to enforce flux continuity and is therefore referred to as the \emph{Lagrange multiplier} (La) system. Finally, the third system uses a Robin-type of interface condition and is thus called the \emph{Robin} (Ro) formulation. 

Each system is presented herein as a variational formulation posed in (subspaces of) spaces of square integrable functions. We assume that the spaces possess sufficient regularity for the (differential) operators in the systems to be well-defined. However, we reserve the precise definitions of these function spaces for a later stage since these require appropriately weighted norms. 

The first formulation follows the classic derivation of \cite{discacciati2002mathematical}. Here, it is assumed that the pressure $p_D$ has sufficient regularity for its trace on $\Gamma$ to be well-defined. The weak form of \eqref{eq:darcy}-\eqref{eq: bcs} yields the \emph{Trace formulation}: Find
$(\vels, p_S, p_D)\in \bV_S \times Q_S \times Q_D$ such that
\begin{equation}\label{eq:ds_weak}
  \begin{aligned}
    (2\mu\symgrad{\vels}, \symgrad{\velstest})_{\Omega_S} 
    + \betatangent(\utangent \cdot \vels, \utangent \cdot \velstest)_{\Gamma}&\\ 
    - (p_S, \nabla\cdot \velstest )_{\Omega_S} 
    + (p_D, \uonormal \cdot \velstest)_{\Gamma}  &= (\bfS, \velstest)_{\Omega_S},
    &\forall \velstest \in \bV_S, \\
    -(\nabla\cdot \vels, q_S)_{\Omega_S} &= 0,
    &\forall q_S \in Q_S, \\
    (\uonormal \cdot \vels, q_D)_{\Gamma} - (\kovermu \nabla p_D, \nabla q_D)_{\Omega_D} &= (f_D, q_D)_{\Omega_D},
    &\forall q_D \in Q_D.
  \end{aligned}    
\end{equation}
Here, and throughout this work, we use $(f, g)_\Sigma := \int_\Sigma f g$. We employ the same notation for vector and tensor-valued functions defined on a domain $\Sigma$.

Problem \eqref{eq:ds_weak} can be naturally discretized
by ($H^1$-)conforming finite element schemes, for example, the lowest order Taylor-Hood ($\vec{P}_2$-$P_1$)
pair for Stokes velocity and pressure and continuous piece-wise quadratic Lagrange ($P_2$) elements for the Darcy pressure ($\vec{P}_2$-$P_1$-$P_2$ in the following).

The second formulation is motivated by cell-centered discretization methods including finite volume methods and non-conforming finite element methods of lowest order. In that case, the trace of $p_D$ is not (directly) available since there is no interfacial degree of freedom and it is common to use a discrete gradient reconstruction scheme to retrieve the interface pressure. 
To illustrate this, let us assume that the Darcy pressure space $Q_D$ consists of piece-wise constant functions.
Introducing $p_\Gamma$ as the unknown interface pressure, a
two-point approximation (TPFA) of the flux on a facet $F \subset \Gamma$ reads
\begin{equation}\label{eq:tpfa_interface}
  -\uonormal \cdot \kovermu \nabla p_D := -\kovermu \frac{p_D|_{K} - p_\Gamma}{h_K}\mbox{ on }F,
\end{equation}
where $p_D|_{K}$ denotes the pressure in the center of the element $K \subseteq \Omega_D$ with $F \subseteq \partial K$ and $h_K$ is the distance between the
centroids of $K$ and $F$. We recall that $\uonormal$ denotes the unit normal outward to $\Omega_S$.
Applying \eqref{eq:tpfa_interface} in \eqref{eq:mass_csrv} yields a
discrete interface condition
\begin{equation}\label{eq:mass_csrv_discrete}
  \uonormal\cdot\vels + \betanormal^{-1} (p_D|_{K} - p_\Gamma) = 0\mbox{ on }F,
\end{equation}
%
%
with $\betanormal := \betanormalfull > 0$. Despite its motivation originating from the discrete case, we shall now consider $\betanormal$ as a model parameter, allowing for a continuous formulation. In particular, we use \eqref{eq:mass_csrv_discrete} to model the flux continuity condition \eqref{eq:mass_csrv} and arrive at the \emph{Lagrange multiplier formulation}:
Find $(\vels, p_S, p_D, p_{\Gamma})\in \bV_S \times Q_S \times Q_D\times \Lambda$
such that
\begin{equation}\label{eq:dsLM_weak}
  \begin{aligned}
    (2\mu\symgrad{\vels}, \symgrad{\velstest})_{\Omega_S} 
    + \betatangent(\utangent \cdot \vels, \utangent \cdot \velstest)_{\Gamma}&\\ 
    - (p_S, \nabla\cdot \velstest )_{\Omega_S} 
    + (p_\Gamma, \uonormal \cdot \velstest)_{\Gamma}  &= (\bfS, \velstest)_{\Omega_S},
    &\forall \velstest &\in \bV_S, \\
    -(\nabla\cdot \vels, q_S)_{\Omega_S} &= 0,
    &\forall q_S &\in Q_S, \\
    - \left(\kovermu \nabla p_D, \nabla q_D \right)_{\Omega_D} - \left( \betanormal^{-1} \left(p_D-p_{\Gamma}\right), q_D\right)_{\Gamma}&= (f_D, q_D)_{\Omega_D},
    &\forall q_D &\in Q_D,\\
    (\uonormal \cdot \vels, q_{\Gamma})_{\Gamma} + \left(\betanormal^{-1} \left(p_D-p_{\Gamma}\right), q_{\Gamma}\right)_{\Gamma} &= 0,
    &\forall q_{\Gamma} &\in \Lambda.\\
  \end{aligned}
\end{equation}
By construction, this formulation is tailored for discretization methods that use cell-centered pressure variables. The precise discretization of the second-order terms $(2\mu\symgrad{\vels}, \symgrad{\velstest})_{\Omega_S}$ and $\left(\kovermu \nabla p_D, \nabla q_D \right)_{\Omega_D}$ is presented
in \cref{sec:discrete_ops}. Furthermore, we emphasize that only the specific choice of $\betanormal = \betanormalfull$ leads to a discretization scheme that is consistent with \eqref{eq:darcy}-\eqref{eq: bcs}.

Our third and final formulation is obtained by eliminating the Lagrange multiplier. 
For that, we once again consider a facet $F$ with an adjacent cell $K \subseteq \Omega_D$. The combination of the momentum balance \eqref{eq:stress} with condition \eqref{eq:mass_csrv_discrete} yields a Robin-type interface condition
\begin{equation}\label{eq:robin}
  - \uonormal \cdot \stressvar \cdot \uonormal 
  = p_D|_{K} + \betanormal \vels \cdot \uonormal \mbox{ on }F.
\end{equation}
By using \eqref{eq:robin} to model flux continuity, we arrive at
the \emph{Robin formulation}: Find $(\vels, p_S, p_D)\in \bV_S \times Q_S \times Q_D$ such that
\begin{equation}\label{eq:ds_robin_weak}
  \begin{aligned}
    (2\mu\symgrad{\vels}, \symgrad{\velstest})_{\Omega_S}
    + \betatangent (\utangent \cdot \vels, \utangent \cdot \velstest)_{\Gamma}& \\
    + \betanormal (\uonormal \cdot \vels, \uonormal \cdot \velstest)_{\Gamma}& \\  
    - (p_S, \nabla\cdot \velstest )_{\Omega_S} 
    + (p_D, \uonormal \cdot \velstest)_{\Gamma}  &= (\bfS, \velstest)_{\Omega_S}, & 
    \forall \velstest &\in \bV_S,\\
    - (\nabla\cdot \vels, q_S)_{\Omega_S} &= 0,
    &\forall q_S &\in Q_S, \\
    (\uonormal \cdot \vels, q_D)_{\Gamma} - (\kovermu \nabla p_D, \nabla q_D)_{\Omega_D} &= (f_D, q_D)_{\Omega_D},
    &\forall q_D &\in Q_D.\\
  \end{aligned}    
\end{equation}
Similar to \eqref{thm:dsLM}, this formulation is amenable to cell-centered finite volume or non-conforming finite element methods. We emphasize that, although variational formulations are more common for finite element practitioners, these final two systems can be interpreted term by term using finite volume discretization techniques.

\subsection{Motivating example}

Having defined the variational problems, our aim is to construct
parameter-robust solvers for all three formulations. 
By robustness, we mean that the preconditioned system has a bounded eigenvalue spectrum independent of modeling and discretization parameters, in particular $\mu$, $\kovermu$, $\betatangent$, the discretization length $h$, and the Robin coefficient $\betanormal$. We base our approach on operator preconditioning using non-standard, weighted Sobolev spaces.

To illustrate the necessity of these techniques, let us first
illustrate that a na{\"i}ve but seemingly sensible approach in standard norms does not yield parameter-robustness.
More precisely, in \cref{ex:naive} we show that natural norms of the solution
spaces of the coupled Stokes-Darcy problem do not translate to robust
preconditioners.

\begin{example}[Standard norm preconditioner]\label{ex:naive}
  We consider the Trace formulation \eqref{eq:ds_weak} on $\Omega_S=\left[0, 1\right]\times \left[1, 2\right]$
  and $\Omega_D=\left[0, 1\right]\times \left[0, 1\right]$ with the source
  terms $\bfS$, $f_D$ defined in \eqref{eq:mms_rhs} and artificially balanced (non-zero right-hand side) coupling conditions \eqref{eq:mms_coupling} (see \cref{sec:mms} for details).
  We let $\Gamma^{\vec{u}}_S$ be the top edge of $\Omega_S$ while the bottom
  edge of $\Omega_D$ is $\Gamma^p_D$. On the remaining parts of the boundaries, Neumann
  boundary conditions are assumed, i.e. traction for the Stokes and normal flux 
  for the Darcy problem. The boundary conditions are non-homogenerous
  with the data based on the manufactured exact solution \eqref{eq:mms}.

  Since \eqref{eq:ds_weak} with the above boundary conditions is well-posed
in $\bV_S=\bH_{0, \Gamma^{\vec{u}}_S}^{1}(\Omega_D)$,
  $Q_S=L^2(\Omega_D)$, and $Q_D=H_{0, \Gamma^p_D}^1(\Omega_D)$ (see \cite{discacciati2002mathematical}), we may want to consider as
  preconditioner the block-diagonal operator
  \begin{equation}\label{eq:naive}
\mathcal{B} := \begin{bmatrix}
  -\nabla\cdot(2\mu\symgradop) + \betatangent\tracetangentprime\tracetangent & & \\
   & (2\mu)^{-1}I& \\
  & & -\kovermu\Delta\\
\end{bmatrix}^{-1},
  \end{equation}
  where $\tracetangent:\bV_S\rightarrow \bV_S'$ is the tangential trace operator.
  We remark that \eqref{eq:naive} is the Riesz map with respect to the parameter-weighted
  inner products of $\bV_S\times Q_S\times Q_D$, which for $\mu=1$, $\perm=1$,
  $\betatangent = 0$, reduce to standard inner products of the spaces. In particular,
  for the first block of \eqref{eq:naive}, we recall that the first Korn inequality holds
  as $\lvert \Gamma^{\vec{u}}_S \rvert>0$. 
  
  Using discrete spaces $\bV_{S, h}\subseteq \bV_S $, $Q_{S, h}\subseteq Q_S$, $Q_{D, h}\subseteq Q_D $
  constructed respectively with $\vec{P}_2$, $P_1$ and $P_2$ elements 
  we investigate robustness of \eqref{eq:naive}
  by considering boundedness of preconditioned MinRes iterations with mesh
  refinement and parameter variations. The iterative solver
  is started from
  an initial vector representing a random function in $\bV_{S, h}\times Q_{S, h}\times Q_{D, h}$ (implying that the
  degrees of freedom (dofs) associated with the Dirichlet boundary conditions are
  set to $0$, while the remaining dofs are drawn randomly from $\left[0, 1\right)$) 
  and terminates once the preconditioned residual norm is reduced by factor
  10\textsuperscript{8}. The preconditioner is computed by LU decomposition.
  
  In \cref{tab:bad_precond}, we report the number of MinRes iterations
  required to satisfy the convergence criteria. We observe that the iterations
  are stable in mesh size. However, there is a clear dependence on permeability
  and the iterations grow with decreasing $\perm$. The BJS parameter (or $\betatangent$) seems
  to have little effect on the solver convergence.

  {We conclude that even though
    the blocks of \eqref{eq:naive} define parameter-robust preconditioners for the individual
    Stokes and Darcy subproblems this property is not sufficient for parameter-robustness
    in the coupled Stokes-Darcy problem.
  }
  
\begin{table}
  \centering
  \caption{
    Performance of block diagonal preconditioner \eqref{eq:naive} for
    the Trace formulation of Stokes-Darcy problem \eqref{eq:ds_weak} discretized by
    $\vec{P}_2$-$P_1$-${P}_2$ elements. Setup of \cref{ex:naive}
    and $\mu=1$. Discretization length scale is denoted by $h$.
    This na{\"i}ve preconditioner is sensitive to variations
    in permeability $\perm$.
  }
  \label{tab:bad_precond}
  \vspace{-10pt}
  \footnotesize{
    \begin{tabular}{c|cccc||cccc}
      \hline
      & \multicolumn{4}{c||}{$\alpha=1$} & \multicolumn{4}{c}{$\alpha=0$}\\
      \hline
      \backslashbox{$\perm$}{$h$} & $2^{-4}$ & $2^{-5}$ & $2^{-6}$ & $2^{-7}$ & $2^{-4}$ & $2^{-5}$ & $2^{-6}$ & $2^{-7}$\\
      \hline
$1$         & $34$ & $33$  & $32$  & $32$   & $34$  & $33$  & $32$  & $32$\\       
$10^{-1}$    & $39$  & $39$  & $39$  & $37$   & $39$  & $41$  & $39$  & $39$\\       
$10^{-2}$    & $52$  & $52$  & $50$  & $49$   & $55$  & $56$  & $54$  & $53$\\       
$10^{-3}$    & $84$  & $84$  & $82$  & $82$   & $95$  & $90$  & $89$  & $88$\\       
$10^{-4}$    & $165$ & $186$ & $184$ & $187$  & $181$ & $202$ & $205$ & $200$\\   
      \hline
    \end{tabular}
  }
    \end{table}
\end{example}

Following this introductory example in which full parameter-robustness could not be achieved, parameter-robust preconditioners for all
presented formulations of the Stokes-Darcy problem will be constructed using a unified framework introduced next.

\section{Abstract setting}\label{sec:abstract}

We observe that each of the three formulations \eqref{eq:ds_weak}, \eqref{eq:dsLM_weak} and \eqref{eq:ds_robin_weak} presented in \cref{sub: three variational formulations} possesses a symmetric structure. Furthermore, the three systems can be identified as perturbed saddle point problems and we detail this observation in this section. To fully exploit this identification, we present an abstract theory of well-posedness for such problems. After introducing the used notation conventions, the main abstract result is shown and the three systems are each presented and analyzed in this functional framework. 

\subsection{Notation and preliminaries} \label{sub: notation}
We start with an exposition of notation conventions. 
For a bounded domain $\Omega\subset\mathbb{R}^d$, we let $L^2(\Omega)$ denote the space of square integrable functions and let $H^k(\Omega)$, $k\geq 1$ be the usual Sobolev space of functions with integer
derivatives up to order $k$ in $L^2(\Omega)$. 
Homogeneous boundary conditions on $\Gamma \subseteq \partial \Omega$ are indicated using a subscript $0$, i.e. $H_{0, \Gamma}^1(\Omega) := \{ f \in H^1(\Omega) \mid f = 0 \text{ on } \Gamma \}$. 
Vector-valued functions and their corresponding spaces are denoted by bold font. 

The trace space of $H^1(\Omega)$ on $\Gamma$ corresponds to $H^{\frac12}(\Gamma)$, the interpolation space between $L^2(\Gamma)$ and $H^1(\Gamma)$. Its dual is denoted by $H^{-\frac12}(\Gamma)$. More generally, we let $X'$ be the dual of a Hilbert space $X$ and let angled brackets $\langle \cdot, \cdot\rangle_{X', X}$ denote the duality pairing. The subscript on this pairing may be omitted when no confusion arises.

$\lVert \cdot \rVert_{k, \Omega}$ denotes the norm on $H^k(\Omega)$ and $\lVert \cdot \rVert_{\Omega} := \lVert \cdot \rVert_{0, \Omega}$. A weighted space $\alpha X$ with $\alpha > 0$ is endowed with the norm $\| f \|_{\alpha X} := \| \alpha f \|_X$ and its dual is given by $(\alpha X)' = \alpha^{-1} X'$. Moreover, given two Hilbert spaces $X,Y$, the intersection $(X \cap Y)$ and sum $(X + Y)$ form Hilbert spaces endowed with the norms
\begin{align*}
  \| f \|_{X \cap Y}^2 &:= \| f \|_X^2 + \| f \|_Y^2, & 
  \| f \|_{X + Y}^2 &:= \inf_{g \in Y} \left( \| f - g \|_X^2 + \| g \|_Y^2 \right),
\end{align*}
respectively. Moreover, we recall the following relations \cite{bergh2012interpolation}:
\begin{align*}
  (X \cap Y)' &= X' + Y', &
  X \cap (Y_1 + Y_2) &= (X \cap Y_1) + (X \cap Y_2).
\end{align*}

{Finally, the relation $x\lesssim y$ implies that there exists a constant $c>0$, independent of
model parameters, such that $x\leq c y$.}

\subsection{Well-posedness theory of perturbed saddle point problems}

Let $V$ and $Q$ be Hilbert spaces to be specified below. Let $\mathcal{A}: V \times Q \to (V \times Q)'$ be a linear operator of the form
\begin{align} \label{eq: perturbed saddle point}
\mathcal{A} := 
\begin{bmatrix}
  A & B^{\prime} \\
  B & -C
\end{bmatrix},
\end{align}
in which the operators $A$, $B$, and $C$ are subject to the following assumptions:
\begin{subequations} \label{eqs: assumptions}
\begin{itemize}
  \item Let $A: V \to V'$ be such that $\langle Au, v \rangle$ forms an inner product on $V$. We denote the induced norm by
  \begin{align} \label{def: A norm}
    \| v \|_A^2 &:= \langle Av, v \rangle, &
    \forall v &\in V.
  \end{align}
  \item Let $B: V \to Q'$ be a linear operator. Moreover, let $| \cdot |_B$ be a semi-norm on $Q$ such that two constants $\zeta_0, \zeta_\infty \in \mathbb{R}$ exist with
  \begin{align} \label{eq: inf-sup B}
    0 < \zeta_0
    &\le \sup_{v \in V} \frac{\langle Bv, p \rangle}{\| v \|_A | p |_B}
    \le \zeta_\infty
    < \infty, &
    \forall p &\in Q.
  \end{align}
  We refer to $\zeta_0$ as the inf-sup constant and $\zeta_\infty$ as the continuity constant.
  \item Let $C: Q \to Q'$ be such that $\langle Cp, q \rangle$ forms a semi-inner product on $Q$. The induced semi-norm is denoted as
  \begin{align} \label{def: C norm}
    | q |_C^2 &:= \langle Cq, q \rangle, &
    \forall q &\in Q.
  \end{align}
  \item 
  Finally, we assume that the following is a proper norm:
  \begin{align} \label{def: e-norm}
      \enorm{(u,p)}^2 := \| u \|_A^2 + | p |_B^2 + | p |_C^2,
  \end{align}
  and we let $V \times Q$ be the space of (pairs of) measurable functions that are bounded in this norm.
\end{itemize}
\end{subequations}

The model problem of interest then reads:
Given $(f, g) \in (V \times Q)'$, find $(u, p) \in V \times Q$ such that 
\begin{align} \label{eq: model problem}
  \mathcal{A} (u, p) = (f, g).
\end{align}

{We note that similar systems were recently analyzed in
  \cite{hong2021new, boon2021robust} but here we exploit the fact that the operator $A$ is coercive on the entire space $V$, instead of on the kernel of $B$.}

\begin{theorem} \label{thm: abstract wellposedness}
    If conditions \eqref{eqs: assumptions} are fulfilled, then the saddle point problem \eqref{eq: model problem} is well-posed in $V \times Q$ endowed with the norm \eqref{def: e-norm}.
\end{theorem}
\begin{proof}
    We first show that $\mathcal{A}$ is continuous. By the Cauchy-Schwarz inequality, we have
    \begin{align*}
      \langle Au, v \rangle
      &\le \| u \|_A \| v \|_A, &
      \forall u,v &\in V, \\
      \langle Cp, q \rangle
      &\le | p |_C | q |_C, &
      \forall p,q &\in Q.
    \end{align*}
    Finally, the existence of $\zeta_\infty$ in \eqref{eq: inf-sup B} ensures that $B$ is continuous. The combination of these three inequalities provides the continuity of $\mathcal{A}$. Since the operator $\mathcal{A}$ is linear and symmetric, it now suffices to show that $\xi$ exists such that
    \begin{align}
        \inf_{(u, p)} \sup_{(v, q)} \frac{\langle \mathcal{A} (u, p), (v, q) \rangle}{\enorm{(u, p)} \enorm{(v, q)} }
        \ge \xi > 0.
    \end{align}
    
    Let $(u, p) \in V \times Q$ be given and let us proceed by constructing a suitable test function $(v, q) \in V \times Q$. First, the inf-sup condition \eqref{eq: inf-sup B} allows us to construct $v^p \in V$ such that 
    \begin{align} \label{eq: properties v^p}
        \langle Bv^p, p \rangle &= | p |_B^2, &
        \zeta_0 \| v^p \|_A &\le | p |_B.
    \end{align}
    Now let $(v, q) := (u + \zeta_0^2 v^p, -p)$ with $\zeta_0$ from \eqref{eq: inf-sup B}.
    Substituting these definitions, we obtain:
    \begin{align} \label{eq: intermediate 1}
        \langle \mathcal{A} (u, p), (v, q) \rangle
        &= \langle Au, v \rangle + \langle Bv, p \rangle + \langle Bu, q \rangle - \langle Cp, q \rangle \nonumber\\
        &= \| u \|_A^2 + \langle Au, \zeta_0^2 v^p \rangle + \langle B \zeta_0^2 v^p, p \rangle + | p |_C^2 \nonumber\\
        &= \| u \|_A^2 + \zeta_0^2 \langle Au, v^p \rangle + \zeta_0^2 | p |_B^2 + | p |_C^2.
    \end{align}
    Next, we need to bound the second term in the right-hand side from below. Using the Cauchy-Schwarz inequality, the inequality $2ab \le a^2 + b^2$, and \eqref{eq: properties v^p}, we derive
    \begin{align} \label{eq: intermediate 2}
        \zeta_0^2 \langle Au, v^p \rangle 
        &\ge - \zeta_0^2 \| u \|_A \| v^p \|_A \nonumber\\
        &\ge - \frac12 \| u \|_A^2 - \frac12 \zeta_0^4 \| v^p \|_A^2 \nonumber\\
        &\ge - \frac12 \| u \|_A^2 - \frac12 \zeta_0^2 | p |_B^2.
    \end{align}
    In turn, \eqref{eq: intermediate 1} and \eqref{eq: intermediate 2} imply
    \begin{align} \label{eq: lower bound}
        \langle \mathcal{A} (u, p), (v, q) \rangle
        &\ge \frac12 \| u \|_A^2 + \frac12 \zeta_0^2 | p |_B^2 + | p |_C^2 
        \ge 
        \frac{\min \{ 1, \zeta_0^2 \}}{2} \enorm{(u, p)}^2.
    \end{align}

    Next, we show that $(v, q)$ is bounded in the norm \eqref{def: e-norm} by $(u, p)$:
    \begin{align} \label{eq: upper bound}
        \enorm{(v, q)}^2 
        &= \| u + \zeta_0^2 v^p \|_A^2 + | p |_B^2 + | p |_C^2  \nonumber \\
        &\le 2 \| u \|_A^2 + 2 \| \zeta_0^2 v^p \|_A^2 + | p |_B^2 + | p |_C^2  \nonumber \\
        &\le 2 \| u \|_A^2 + (1 + 2 \zeta_0^2) | p |_B^2 + | p |_C^2 \nonumber \\
        &\le \max\{2, 1 + 2 \zeta_0^2\} \enorm{(u, p)}^2.
    \end{align}

    Combining \eqref{eq: lower bound} and \eqref{eq: upper bound}, we have 
    \begin{align*}
        \frac{\langle \mathcal{A} (u, p), (v, q) \rangle}{\enorm{(u, p)} \enorm{(v, q)} }
        \ge 
        \frac{\min \{ 1, \zeta_0^2 \}}{2 \sqrt{\max\{2, 1 + 2 \zeta_0^2\}}} > 0.
    \end{align*}
    The Banach-Ne\v{c}as-Babu\v{s}ka theorem \cite[Thm. 2.6]{ern2013theory} then provides the result.
\end{proof}
{
  We remark that \cref{thm: abstract wellposedness} is related to the well-posedness
  theory of abstract saddle-point systems with penalties \cite{braess1996stability}.
}

At this point, we are ready analyze the individual
Stokes-Darcy formulations within the introduced abstract framework.

\subsection{The Trace formulation}\label{sec:ds_standard}
Observe that the left-hand side of \eqref{eq:ds_weak} defines an operator $\mathcal{A}^\text{Tr}$
on $\bV_{S}\times \left(Q_{S}\times Q_{D}\right)$
\begin{equation}\label{eq:daq_op}
\mathcal{A}^\text{Tr} := \begin{bmatrix}
  -\nabla\cdot(2\mu\symgradop) + \betatangent \tracetangentprime\tracetangent & \nabla & \tracenormalprime\\
  -\nabla\cdot & & \\
  \tracenormal & & \kovermu\Delta
\end{bmatrix},
\end{equation}
where $\tracetangent: \bV_S\rightarrow V^{\prime}_S$ is the tangential trace operator,
and $\tracenormal: \bV_S\rightarrow Q^{\prime}_D$ is the normal trace operator. 
The system thus fits template \eqref{eq: perturbed saddle point} of a perturbed saddle-point problem with
\begin{align*}
  \langle A \vels, \velstest \rangle &:= 
  (2\mu\symgrad{\vels}, \symgrad{\velstest})_{\Omega_S}
  + \betatangent(\utangent \cdot \vels, \utangent \cdot \velstest)_{\Gamma} , \\
  \langle B \vels, (q_S, q_D) \rangle &:=
  - (\nabla\cdot \vels, q_S)_{\Omega_S}
  + (\uonormal \cdot \vels, q_D)_{\Gamma}, \\
  \langle C (p_S, p_D), (q_S, q_D) \rangle &:=
  (\kovermu \nabla p_D, \nabla q_D)_{\Omega_D}.
\end{align*}
Based on these operators, we define the following (semi-)norms:
\begin{subequations}
\begin{align}
    \| \vels \|_A^2 &:= 2\mu \| \symgrad{\vels} \|_{\Omega_S}^2
    + \betatangent \| \utangent \cdot \vels \|_\Gamma^2, \\
    | (p_S, p_D) |_B^2 &:= 
    (2\mu)^{-1} \| p_S \|_{\Omega_S}^2 
    + (2\mu)^{-1} \| p_D \|_{-\frac12, \Gamma}^2, \\
    | (p_S, p_D) |_C^2 &:= \kovermu \| \nabla p_D \|_{\Omega_D}^2.
\end{align}
\end{subequations}
Finally, in the context of \cref{thm: abstract wellposedness}, we
consider the following norm
\begin{equation}\label{eq:th_precond_norm}
    \begin{aligned}
        \enorm{ (\vels, p_S, p_D) }^2 &:= 
        2\mu \| \symgrad{\vels} \|_{\Omega_S}^2
        + \betatangent \| \utangent \cdot \vels \|_\Gamma^2 \\
        &\quad + (2\mu)^{-1} \| p_S \|_{\Omega_S}^2
        + (2\mu)^{-1} \| p_D \|_{-\frac12, \Gamma}^2
        + \kovermu \| \nabla p_D \|_{\Omega_D}^2.
    \end{aligned}
\end{equation}
%

\begin{theorem} \label{thm: Taylor Hood well-posed}
  Problem \eqref{eq:ds_weak} is well-posed in $V \times Q$ endowed with the norm \eqref{eq:th_precond_norm}.
\end{theorem}
\begin{proof}
    We follow the assumptions of \cref{thm: abstract wellposedness}. First, the properties \eqref{def: A norm} and \eqref{def: C norm} are immediately fulfilled. 
    Next, the continuity of $B$ is shown by the following calculation, utilizing the Cauchy-Schwarz inequality and a trace inequality:
    \begin{align*}
      \langle B \vels, (p_S, p_D) \rangle &=
      - (\nabla \cdot \vels, p_S)_{\Omega_S} + 
      (\uonormal \cdot \vels, p_D)_{\Gamma} \\
      &\le \| \nabla \cdot \vels \|_{\Omega_S} \| p_S \|_{\Omega_S} + 
      \| \uonormal \cdot \vels \|_{\frac12, \Gamma} \| p_D \|_{-\frac12, \Gamma} \\
      &\lesssim \| \symgrad{\vels} \|_{\Omega_S} 
      \left(\| p_S \|_{\Omega_S} + \| p_D \|_{-\frac12, \Gamma} \right) \\
      &\lesssim (2\mu)^{\frac12} \| \symgrad{\vels} \|_{\Omega_S} 
      (2\mu)^{-\frac12} \left(\| p_S \|_{\Omega_S}^2 + \| p_D \|_{-\frac12, \Gamma}^2 \right)^{\frac12} \\
      &= \| \vels \|_A | (p_S, p_D) |_B.
    \end{align*}

    The inf-sup condition of $B$ is considered next. Let $(p_S, p_D)$ be given. Let $\vec{v}^{p_S} \in \vec{H}^1(\Omega_S)$ be constructed, using the Stokes inf-sup condition, such that
    \begin{subequations}\label{eq: test function vps}
    \begin{align} 
            \vec{v}^{p_S}|_{\Gamma} &= 0, &
            \nabla \cdot \vec{v}^{p_S} &= -p_S, \\
            \| \symgrad{\vec{v}^{p_S}} \|_{\Omega_S} 
            &\lesssim \| p_S \|_{\Omega_S}.
        \end{align}
        \end{subequations}
    On the other hand, let $\phi \in H^{\frac12}(\Gamma)$ be the Riesz representative of $p_D|_\Gamma \in H^{-\frac12}(\Gamma)$. We then define $\vec{v}^{p_D} \in \vec{H}^1(\Omega_S)$ as the bounded extension that satisfies
    \begin{subequations} \label{eq: test function vpd}
    \begin{align}
        \vec{v}^{p_D}|_{\Gamma} &= \phi \uonormal, &
        \nabla \cdot \vec{v}^{p_D} &= 0, \\
        \| \symgrad{\vec{v}^{p_D}} \|_{\Omega_S} 
        &\lesssim \| \phi \|_{\frac12, \Gamma} 
        = \| p_D \|_{-\frac12, \Gamma}.
    \end{align}
    \end{subequations}
    We are now ready to set the test function $\velstest := (2\mu)^{-1} (\vec{v}^{p_S} + \vec{v}^{p_D})$. Noting that $\utangent \cdot \velstest = 0$ on $\Gamma$, this function satisfies
    \begin{subequations}
    \begin{align}
        \langle B \velstest, (p_S, p_D) \rangle
        &= - (2\mu)^{-1}(\nabla \cdot \vec{v}^{p_S}, p_S)_{\Omega_D} + 
        (2\mu)^{-1}(\uonormal \cdot \vec{v}^{p_D}, p_D)_{\Gamma} \nonumber\\
        &= (2\mu)^{-1} \| p_S \|_{\Omega_S}^2
        + (2\mu)^{-1} \| p_D \|_{-\frac12, \Gamma}^2 \nonumber\\
        &= | (p_S, p_D) |_B^2, \\
        \| \velstest \|_A
        &= (2\mu)^{\frac12} \| \symgrad{(2\mu)^{-1}(\vec{v}^{p_S} + \vec{v}^{p_D})} \|_{\Omega_S} \nonumber\\
        &\le (2\mu)^{-\frac12} \left( \| \symgrad{\vec{v}^{p_S}} \|_{\Omega_S}
        + \| \symgrad{\vec{v}^{p_D}} \|_{\Omega_S} \right) \nonumber\\
        &\lesssim (2\mu)^{-\frac12} \left( \| p_S \|_{\Omega_S}
        + \| p_D \|_{-\frac12, \Gamma} \right) \nonumber\\
        &\lesssim | (p_S, p_D) |_B.
    \end{align}
    \end{subequations}
    Hence, condition \eqref{eq: inf-sup B} is fulfilled.

    Finally, it is straightforward to verify that \eqref{def: e-norm} is a norm on $Q$ and thus the assumptions of \cref{thm: abstract wellposedness} are fulfilled.
\end{proof}

Following operator preconditioning \cite{Mardal2011}, and using the well-posedness
result of \cref{thm: Taylor Hood well-posed}, a preconditioner
for the Stokes-Darcy in the Trace formulation \eqref{eq:ds_weak} is the Riesz
map with respect to the inner product inducing the norms \eqref{eq:th_precond_norm},
that is, the block diagonal operator
%
\begin{equation}\label{eq:daq_precond}
\mathcal{B}^\text{Tr} := \begin{bmatrix}
  -\nabla\cdot(2\mu\symgradop) + \betatangent\tracetangentprime\tracetangent & & \\
   & (2\mu)^{-1}I& \\
  & & -\kovermu\Delta + (2\mu)^{-1}\left(-\Delta_{\Gamma}\right)^{-1/2}\\
\end{bmatrix}^{-1}.
\end{equation}
%
Here, the subscript $\Gamma$ signifies that the fractional operator acts
on the interface. We demonstrate numerically, robustness of the preconditioner \eqref{eq:daq_precond}
using both $H^1$-conforming and non-conforming Stokes-Darcy-stable elements in \cref{sec:numeric}. Here, we continue with the remaining two formulations
concerning cell-centered finite volume schemes and lowest-order non-conforming finite element schemes.

\subsection{The Lagrange multiplier formulation}\label{sec:ds_LM}
Variational problem \eqref{eq:dsLM_weak} defines an operator $\mathcal{A}^\text{La}$ on $\bV_{S}\times \left(Q_{S}\times Q_{D}\times\Lambda\right)$
\begin{equation}\label{eq:dsLM_op}
\mathcal{A}^\text{La} := \begin{bmatrix}
  -\nabla\cdot( 2\mu\symgradop ) + \betatangent\tracetangentprime\tracetangent & \nabla & & \tracenormalprime\\
  -\nabla\cdot & & & \\
  & & \kovermu\Delta - \betanormal^{-1} T^{\prime}T & \betanormal^{-1} T^{\prime}\\
  \tracenormal & & \betanormal^{-1} T & -\betanormal^{-1} I
\end{bmatrix},
\end{equation}
where $T: Q_{D}\rightarrow \Lambda'$ is a trace/restriction operator
for the Darcy pressure space. We observe that the lower $2\times 2$ block forms a discretization
of the Laplacian in terms of the interior Darcy pressure and the interface pressure.
In this sense, \eqref{eq:dsLM_op} is similar to \eqref{eq:daq_op}.

We note that the operator \eqref{eq:dsLM_op} also fits template \eqref{eq: perturbed saddle point} with the operators given by
\begin{align*}
  \langle A \vels, \velstest \rangle &:= 
  (2\mu\symgrad{\vels}, \symgrad{\velstest})_{\Omega_S}
  + \betatangent(\utangent \cdot \vels, \utangent \cdot \velstest)_{\Gamma}, \\
  \langle B \vels, (q_S, q_D, q_\Gamma) \rangle &:=
  - (\nabla\cdot \vels, q_S)_{\Omega_S}
  + (\uonormal \cdot \vels, q_\Gamma)_{\Gamma}, \\
  \langle C (p_S, p_D, p_\Gamma), (q_S, q_D, q_\Gamma) \rangle &:=
  (\kovermu \nabla p_D, \nabla q_D)_{\Omega_D}
  + (\betanormal^{-1} (p_D - p_\Gamma), (q_D - q_\Gamma))_{\Gamma}.
\end{align*}
These operators lead us to the following norms
\begin{subequations} \label{eqs: norms LM}
\begin{align}
    \| \vels \|_A^2 &:= 2\mu \| \symgrad{\vels} \|_{\Omega_S}^2
    + \betatangent \| \utangent \cdot \vels \|_\Gamma^2, \\
    | (p_S, p_D, p_\Gamma) |_B^2 &:= 
    (2\mu)^{-1} \| p_S \|_{\Omega_S}^2
    + (2\mu)^{-1} \| p_\Gamma \|_{-\frac12, \Gamma}^2, \\
    | (p_S, p_D, p_\Gamma) |_C^2 &:= \kovermu \| \nabla p_D \|_{\Omega_D}^2 
    + \betanormal^{-1} \| p_D - p_\Gamma \|_\Gamma^2.
    \end{align}
\end{subequations}

\begin{theorem}\label{thm:dsLM}
  Problem \eqref{eq:dsLM_op} is well-posed in $V \times Q$ endowed with the energy norm \eqref{def: e-norm} formed by \eqref{eqs: norms LM}.
\end{theorem}
\begin{proof}
  Assumptions \eqref{def: A norm} and \eqref{def: C norm} are again immediate.
  Assumption \eqref{eq: inf-sup B} was proven in \cref{thm: Taylor Hood well-posed}
  (with $p_D|_{\Gamma}$ substituted for $p_\Gamma$). Then, \cref{thm: abstract wellposedness}
  provides the result.
\end{proof}

Following \cref{thm:dsLM}, a preconditioner for problem \eqref{eq:dsLM_weak}
reads
\begin{equation}\label{eq:dsLM_precond}
\mathcal{B}^\text{La} := \begin{bmatrix}
  \scriptstyle -\nabla\cdot(2\mu\symgradop) + \betatangent\tracetangentprime\tracetangent &  & & \\
   & \scriptstyle (2\mu)^{-1} I & & \\
   & & \scriptstyle -\kovermu\Delta +\betanormal^{-1} T^\prime T & \scriptstyle -\betanormal^{-1} T'\\
   & & \scriptstyle -\betanormal^{-1} T & \scriptstyle \betanormal^{-1} I_{\Gamma} + (2\mu)^{-1}(-\Delta_{\Gamma})^{-1/2}
\end{bmatrix}^{-1}.
\end{equation}

Finally, we consider the third variational form established by eliminating the
Lagrange multiplier on the coupling interface.

\subsection{The Robin formulation}\label{sec:ds_robin}
We observe that problem \eqref{eq:ds_robin_weak} is given in terms
of the operator $\mathcal{A}^\text{Ro}$ on $\bV_{S}\times \left(Q_{S}\times Q_{D}\right)$
\begin{equation}\label{eq:fvm_op}
\mathcal{A}^\text{Ro} := \begin{bmatrix}
  -\div (2\mu\symgradop) 
  + \betatangent \tracetangentprime\tracetangent 
  + \betanormal \tracenormalprime\tracenormal 
  & \nabla & \tracenormalprime\\
  -\nabla\cdot & & \\
  \tracenormal & & \kovermu\Delta
\end{bmatrix}.
\end{equation}
%
Note again that \eqref{eq:fvm_op} has the structure \eqref{eq: perturbed saddle point} with the operators given by
\begin{align*}
  \langle A \vels, \velstest \rangle &:= 
  (2\mu\symgrad{\vels}, \symgrad{\velstest})_{\Omega_S} \nonumber \\
    &\quad + \betatangent(\utangent \cdot \vels, \utangent \cdot \velstest)_{\Gamma} 
    + \betanormal(\uonormal \cdot \vels, \uonormal \cdot \velstest)_{\Gamma}, \\
  \langle B \vels, (q_S, q_D) \rangle &:=
  - (\nabla\cdot \vels, q_S)_{\Omega_S}
  + (\uonormal \cdot \vels, q_D)_{\Gamma}, \\
  \langle C (p_S, p_D), (q_S, q_D) \rangle &:=
  (\kovermu \nabla p_D, \nabla q_D)_{\Omega_D}.
\end{align*}
In the framework of \cref{thm: abstract wellposedness}, we identify the following norms:
\begin{subequations} \label{eqs: norms fvm}
\begin{align}
    \| \vels \|_A^2 &:= 2\mu \| \symgrad{\vels} \|_{\Omega_S}^2
    + \betatangent \| \utangent \cdot \vels \|_\Gamma^2
    + \betanormal \| \uonormal \cdot \vels \|_\Gamma^2, \\
    | (p_S, p_D) |_B^2 &:= 
    (2\mu)^{-1} \| p_S \|_{\Omega_S}^2
    + \| p_D \|_{(2\mu)^{-\frac12}H^{-\frac12}(\Gamma) + \betanormal^{-\frac12}L^2(\Gamma)}^2, \\
    | (p_S, p_D) |_C^2 &:= \kovermu \| \nabla p_D \|_{\Omega_D}^2.
    \end{align}
\end{subequations}

We remark that the control on the pressure variable is weakened in comparison with \eqref{eqs: norms LM}. This is a direct result from the fact that $\vels$ is now in a smaller space (with a stronger norm).

\begin{theorem}\label{thm:ds_robin}
  Problem \eqref{eq:fvm_op} is well-posed in $V \times Q$ endowed with the energy norm formed by \eqref{eqs: norms fvm}.
\end{theorem}
\begin{proof}
  Assumptions \eqref{def: A norm} and \eqref{def: C norm} follow immediately. We continue with the bounds on $B$, starting with continuity:
  \begin{align*}
    \langle B \vels, (q_S, q_D) \rangle &=
    - (\nabla\cdot \vels, q_S)_{\Omega_S}
    + (\uonormal \cdot \vels, q_D)_{\Gamma} \\
    &\le \| \symgrad{\vels} \|_{(2\mu)^{\frac12} L^2(\Omega_S)} \| q_S \|_{(2\mu)^{-\frac12} L^2(\Omega_S)} \\
    &\quad + \| \uonormal \cdot \vels \|_{(2\mu)^{\frac12}H^{\frac12}(\Gamma) \cap \betanormal^{\frac12}L^2(\Gamma)} 
    \| q_D \|_{(2\mu)^{-\frac12}H^{-\frac12}(\Gamma) + \betanormal^{-\frac12}L^2(\Gamma)} \\
    &\lesssim \| \vels \|_A + | (q_S, q_D) |_B,
  \end{align*}
  in which we used that $(c X)' = c^{-1} X'$ for $c > 0$ and $(X \cap Y)' = X + Y$, cf.~\cref{sub: notation}.

  Secondly, we prove the inf-sup condition for which we follow the same approach as in \cref{thm: Taylor Hood well-posed}. Let $(p_S, p_D)$ be given with bounded $B$-norm and let $\vec{v}^{p_S}$ satisfy \eqref{eq: test function vps}. For notational convenience, we define
  \begin{align}
    \Lambda &:= (2\mu)^{\frac12}H^{\frac12}(\Gamma) \cap \betanormal^{\frac12}L^2(\Gamma), &
    \Lambda' &:= (2\mu)^{-\frac12}H^{-\frac12}(\Gamma) + \betanormal^{-\frac12}L^2(\Gamma).
  \end{align}
  Now, let $\phi \in \Lambda$ be the Riesz representative of $p_D|_\Gamma \in \Lambda'$. Since $\Lambda \subseteq H^{\frac12}(\Gamma)$, we can define $\vec{v}^{p_D} \in \bH^1(\Omega_S)$ according to \eqref{eq: test function vpd}. This function satisfies the bound 
  $\| \symgrad{\vec{v}^{p_D}} \|_{\Omega_S}
      \lesssim 
      \| \phi \|_{\frac12, \Gamma}$ and we obtain
  \begin{align*}
      \| (2\mu)^{\frac12} \symgrad{\vec{v}^{p_D}} \|_{\Omega_S}^2
      +
      \| \betanormal^{\frac12} \uonormal \cdot \vec{v}^{p_D} \|_{\Gamma}^2 
      &\lesssim 
      \| (2\mu)^{\frac12} \phi \|_{\frac12, \Gamma}^2
      + \| \betanormal^{\frac12} \phi \|_{\Gamma}^2 \\
      &= \| \phi \|_{\Lambda}^2
      = \| p_D \|_{\Lambda'}^2.
  \end{align*}
    
  Finally, we define the test function $\vec{v} = (2\mu)^{-1} \vec{v}^{p_S} + \vec{v}^{p_D}$ and deduce
    \begin{subequations} \label{eqs: v properties}
    \begin{align}
        \langle B \velstest, (p_S, p_D) \rangle
        &= - (2\mu)^{-1}(\nabla \cdot \vec{v}^{p_S}, p_S)_{\Omega_D} + 
        (\uonormal \cdot \vec{v}^{p_D}, p_D)_{\Gamma} \nonumber\\
        &= (2\mu)^{-1} \| p_S \|_{\Omega_S}^2
        + \| p_D \|_{\Lambda'}^2 \nonumber\\
        &= | (p_S, p_D) |_B^2, \\
        \| \velstest \|_A^2
        &= (2\mu) \| \symgrad{(2\mu)^{-1}\vec{v}^{p_S} + \vec{v}^{p_D}} \|_{\Omega_S}^2
        + \betanormal \| \uonormal \cdot \vec{v}^{p_D} \|_\Gamma^2 \nonumber\\
        &\lesssim (2\mu)^{-1} \| \symgrad{\vec{v}^{p_S}} \|_{\Omega_S}^2
        + \| (2\mu)^{\frac12} \symgrad{\vec{v}^{p_D}} \|_{\Omega_S}^2
        + \| \betanormal^{\frac12} \uonormal \cdot \vec{v}^{p_D} \|_{\Gamma}^2 \nonumber\\
        &\lesssim (2\mu)^{-\frac12} \| p_S \|_{\Omega_S}^2
        + \| p_D \|_{\Lambda'}^2 \nonumber\\
        &= | (p_S, p_D) |_B^2.
    \end{align}
    \end{subequations}
    Now, \eqref{eqs: v properties} implies that assumption \eqref{eq: inf-sup B} is fulfilled and the result follows by \cref{thm: abstract wellposedness}.
\end{proof}
\Cref{thm:ds_robin} leads us to the preconditioner for the third formulation:
\begin{equation}\label{eq:ds_robin_precond}
\mathcal{B}^\text{Ro} = \begin{bmatrix}
  \left(\substack{-\nabla\cdot(2\mu\symgradop) 
  + \betatangent \tracetangentprime\tracetangent \\
  + \betanormal \tracenormalprime\tracenormal }\right)^{-1} & & \\
   & \left((2\mu)^{-1} I \right)^{-1}& \\
  & & \left(\substack{\left(-\kovermu\Delta + \betanormal^{-1} I_\Gamma \right)^{-1} \\+
  \left(-\kovermu\Delta + (2\mu)^{-1}\left(-\Delta_{\Gamma}\right)^{-\frac12} \right)^{-1}} \right)\\
\end{bmatrix}.
\end{equation}
Note that the ($1\times1$) Darcy pressure block of the preconditioner contains a sum of two inverse operators.
This construction is typical in preconditioning sums of spaces, as discussed in \cite{baerland2020observation}.

\section{Numerical experiments}\label{sec:numeric}


For finite element and finite volume discretizations, we let $\Omega_{i, h}$, $i=D, S$
($h$ being the characteristic discretization length) denote
the meshes of $\Omega_D$, $\Omega_S$ that conform to $\Gamma$ in the sense that every
facet $F$ on the interface $\Gamma$ satisfies $F = \partial K_D \cap \partial K_S$ for
some unique cell pair $K_D \in{\Omega}_{D, h}$ and $K_S \in{\Omega}_{S, h}$. The mesh
of the interface (consisting of facets $F$) is denoted by $\Gamma_h$.

Unless stated otherwise, the geometry setup and boundary
data of \cref{ex:naive} are used, i.e.
$\Omega_S=\left[0, 1\right]\times \left[1, 2\right]$, $\Omega_D=\left[0, 1\right]\times \left[0, 1\right]$ with
the source terms defined in \eqref{eq:mms_rhs} and the top edge of $\Omega_S$ and the
bottom edge of $\Omega_D$ designated as Dirichlet boundaries $\Gamma^{\vec{u}}_S$, $\Gamma^p_D$, respectively.
On the remaining boundaries, Neumann conditions are given. The non-homogeneous boundary data
matches \eqref{eq:mms}.
In all examples, the Krylov solver terminates when the preconditioned residual
norm is reduced by a factor $10^8$.

All numerical tests are implemented using the scientific software frameworks
FEniCS\textsubscript{ii}~\cite{fenicsii} (FEM) and DuMu${}^\text{x}$/DUNE~\cite{Kochetal2020Dumux,dunegrid2} (FVM),
where we use PETSc~\cite{Petsc-user-ref}, SLEPc~\cite{SLEPc2005} (FEM) and
Eigen~\cite{Eigen3}, Spectra~\cite{Spectra2019} (FVM) for solving exact and approximate
generalized eigenvalue problems (discrete fractional Laplacian, condition numbers).
The preconditioners are implemented within the abstract linear solver frameworks of
PETSc (FEM) and dune-istl~\cite{duneistl} (FVM).

Since the discretization of the preconditioners is not straightforward
due to the interfacial contributions, we first provide some details
regarding their construction in \cref{sec:discrete_precond}. To demonstrate robustness of the
proposed preconditioners, 
we conduct numerical experiments with large parameter ranges motivated by the practical applications and dimensional analysis discussed in \cref{sec:params}. 
Numerical results are finally presented in \cref{sec:param_sweep}.

\subsection{Discrete preconditioners}\label{sec:discrete_precond}
The (only) non-standard component common to all our Stokes-Darcy preconditioners
is the fractional operator $\mu^{-1}(-\Delta_{\Gamma})^{-1/2}$. Following \cite{kuchta2016preconditioners},
we consider here the approximation based on the spectral definition which
requires solution of the following generalized eigenvalue problem in a discrete space
$V_h=V_h(\Gamma_h)$, $n=\operatorname{dim}(V_h)$: For $1\leq i \leq n$ find
$(u_i, \lambda_i)\in V_h\times\mathbb{R}$ such that
\begin{equation}\label{eq:eigv}
(u_i, v)_{\mu^{-\frac12}H^1(\Gamma)}=\lambda_i (u_i, v)_{\mu^{-\frac12}L^2(\Gamma)},\quad\forall v\in V_h,
\end{equation}
with the orthogonality condition $(u_i, u_j)_{\mu^{-\frac12}L^2(\Gamma)}=\delta_{ij}$. Then, we let
\begin{equation}\label{eq:fract_laplace}
\langle \mu^{-1}(-\Delta_{\Gamma})^{-1/2}u, v \rangle := \sum_{i}\lambda^{-1/2}_i(u_i, u)_{\mu^{-\frac12}L^2(\Gamma)}(u_i, v)_{\mu^{-\frac12}L^2(\Gamma)}, \quad u, v\in V_h.
\end{equation}
We note that \eqref{eq:eigv} is related to the weak formulation of $\mu^{-1}(-\Delta_{\Gamma}+I_{\Gamma})u=\mu^{-1}\lambda u$ in $\Gamma$
with Neumann boundary conditions\footnote{The actual boundary data is irrelevant as it does not enter the operator.} on the boundary $\partial \Gamma$.

Introducing matrices
$\vec{A}_h$ (discrete $\mu^{-1}(-\Delta_{\Gamma}+I_{\Gamma})$ operator), $\vec{M}_h$ (discrete $\mu^{-1}I_{\Gamma}$ operator),
the matrix representation of \eqref{eq:fract_laplace} (with respect
to the basis of $V_h$) reads
\[
\vec{M}_h\vec{U}_h\vec{E}^{-1/2}_h\vec{U}^T_h\vec{M}^T_h,\text{ where } \vec{A}_h\vec{U}_h=\vec{M}_h\vec{U}_h\vec{E}_h\text{ and } \vec{U}^T_h\vec{M}_h\vec{U}_h=\ident_h.
\]
That is, $\vec{E}_h, \vec{U}_h \in\mathbb{R}^{n\times n}$ are the solutions of the eigenvalue problem \eqref{eq:eigv}
with the eigenvalues forming the entries of the diagonal matrix $\vec{E}_h$ and columns of $\vec{U}_h$ being the 
$\vec{M}_h$-orthonormal eigenvectors. 
We remark that for cell-centered finite volume
and $\Pnot$ finite element discretizations $\vec{M}_h$  is a diagonal matrix.

While the eigenvalue problem makes the construction inefficient for large scale
applications, it is suitable for our robustness investigations where, in particular,
we are interested in exact preconditioners.
For large scale applications, scalable realizations of the Darcy pressure preconditioners
in \eqref{eq:daq_precond}, \eqref{eq:dsLM_precond} and \eqref{eq:ds_robin_precond} are,
to the best of the authors' knowledge, yet to be established. However,
efficient solvers for the interfacial component alone, i.e $\mu^{-1}(-\Delta_{\Gamma})^{-1/2}$, 
are known, e.g. \cite{Oswald2007MultilevelNF, fuhrer2021multilevel, bramble2000computational, baerland2019multigrid, stevenson2021uniform}.

Concerning the discretization of \eqref{eq:eigv}, we note that we use the full $H^1$-inner
product, since in \cref{ex:naive} and as assumed in \cref{sec:problem}, the interface $\Gamma$ intersects
Neumann boundaries
(see \cite{galvis2007non, holter2020robust} for discussion of multiplier spaces
in relation to the spaces/boundary conditions on the adjacent subproblems).
For the case of $\Gamma$ intersecting boudaries with Dirichlet conditions, we refer to \cref{sec:meet_dirichlet}.

Finally, let us note that while in the multiplier formulation \eqref{eq:dsLM_precond} the trace
space $V_h$ for \eqref{eq:eigv} is explicit, i.e. $V_h=\Lambda_h$, this
is not the case for the preconditioners for the Trace
and Robin formulations, \eqref{eq:daq_precond} and \eqref{eq:ds_robin_precond}. More precisely, to compute the approximation of
\begin{equation*}
\left(-\kovermu\Delta + (2\mu)^{-1}\left(-\Delta_{\Gamma}\right)^{-1/2} \right)^{-1}
\end{equation*}
a mapping $\Pi_{\Gamma}:Q_{D, h}\rightarrow V_h$ is required. In the following,
$\Pi_{\Gamma}$ is defined as an interpolation operator to $V_h$, where for 
$\vec{P}_2$-$P_1$-$P_2$ discretization $V_h$ is constructed with $P_2$ elements.
When using the Crouzeix-Raviart element in a $\vec{C}\vec{R}_1$-$\Pnot$-$\Pnot$ discretization or the cell-centered finite volume discretization, the space $V_h$ is constructed using $\Pnot$ elements.

We remark that the trace space cannot be chosen arbitrarily. In particular,
for the Taylor-Hood element e.g. the choice of $\Pnot$ for $V_h$ results in
parameter sensitivity of the Trace formulation \eqref{eq:ds_weak} with the
preconditioner \eqref{eq:daq_precond}.

\subsection{Relevant parameter ranges}\label{sec:params}
Having specified the discretizations of preconditioners 
we identify next the parameter regimes for which robustness
is investigated in numerical experiments. We chose the parameter ranges based on a scaling analysis and several real-life applications.

Let $U_0$ be the characteristic Stokes velocity magnitude, $\Delta P_0$ the characteristic
pressure difference in the Stokes domain, and $L_0$ the characteristic length scale.
Introducing the dimensionless quantities $\vels = U_0 \velstilde$, $p_i = \Delta P_0 \tilde{p}_i$, $i=S, D$
and $\nabla(\cdot) = L_0^{-1}\tilde{\nabla}(\cdot)$ we arrive at the re-scaled 
Stokes-Darcy system 
\begin{equation}\label{eq:ds_dless}
\begin{aligned}
\tilde{\nabla}\cdot\left( \text{Re}^{-1}\text{Eu}^{-1} \tilde{\vec{\epsilon}}(\velstilde) + \tilde{p}_S \ident  \right) &= 0 & \text{in} \quad\tilde{\Omega}_S, \\
\tilde{\nabla}\cdot\left( \velstilde \right) &= 0 & \text{in} \quad\tilde{\Omega}_S, \\
\tilde{\nabla}\cdot( - \text{Re}\text{Eu}\text{Da} \tilde{\nabla} \tilde{p}_D ) &= 0 & \text{in} \quad\tilde{\Omega}_D,
\end{aligned}
\end{equation}
with the coupling conditions on $\Gamma$,
\begin{equation}\label{eq:interface_dles}
  \begin{aligned}
\utangent \cdot \tilde{\vec{\epsilon}}(\velstilde) \cdot \uonormal  + \alpha \text{Da}^{-1/2} \utangent \cdot\velstilde &= \vec{0},\\    
\uonormal \cdot \left( \text{Re}^{-1}\text{Eu}^{-1} \tilde{\vec{\epsilon}}(\velstilde) + \tilde{p}_S \ident  \right) \cdot \uonormal + \tilde{p}_D &= 0, \\
\velstilde \cdot \uonormal + \text{Re}\text{Eu}\text{Da} \tilde{\nabla} \tilde{p}_D \cdot \uonormal &= 0. \\
\end{aligned}
\end{equation}
Here we introduced the dimensionless velocity gradient $\tilde{\vec{\epsilon}}(\velstilde) = \tilde{\nabla} \velstilde + \tilde{\nabla}^T \velstilde$, the dimensionless numbers
$\text{Re} := \rho U_0 L_0 \mu^{-1}$, $\text{Eu} := \Delta P_0 \rho_0^{-1} U_0^{-2}$, $\text{Da} := \perm L_0^{-2}$,
and $\rho_0$ denotes a characteristic fluid density. We recognize that our equation system is effectively characterized by
a characteristic free-flow number $S = \text{Re}^{-1}\text{Eu}^{-1} = U_0 \mu L_0^{-1} \Delta P_0^{-1}$, the Darcy number, $\text{Da}$,
and the Beavers-Joseph slip coefficient, $\alpha$. Moreover, by comparing \eqref{eq:ds_dless}-\eqref{eq:interface_dles}
with \eqref{eq:darcy}-\eqref{eq: bcs} we observe that for unit scaling parameters $U_0$, $\Delta P_0$, $\rho_0$ and $L_0$
(as is the case in the manufactured problem \eqref{eq:mms}) we can
interpret $S$, respectively $\text{Da}$ as $\mu$ and $\perm$. To estimate the 
relevant ranges, we consider three examples.

\begin{example}[Channel flow over a regular porous medium in a micro-model]\label{ex:channel}
In \cite{Terzis2019}, water flow in a micro-model with a free-flow channel of height $200~\mu\textnormal{m}$ adjacent to
a regular porous medium is a examined at low Reynolds numbers, $U_0 \approx 0.2 - 0.4~\textnormal{mm}\,\textnormal{s}^{-1}$,
$\mu = 10^{-3}~\textnormal{Pa}\,\textnormal{s}$.
The slip coefficient $\alpha$ is determined as $2.26$. The permeability can be estimated by Poiseuille flow in a bundle of tubes due
to the regular geometry and is in the order of $\perm \approx 10^{-12} - 10^{-8}~\textnormal{m}^2$. Hence, $S \approx 1$,
$\textnormal{Da} \in [4\cdot10^{-4}, 4\cdot10^{0}]$, $\alpha = 2.26$.
\end{example}

\begin{example}[Air channel flow over porous medium box in a wind tunnel]\label{ex:tunnel}
Such a scenario may be modeled by the Stokes-Darcy system if the Reynolds number is
sufficiently small ($\textnormal{Re} < 1$). Assuming a channel width of $0.1~\textnormal{m}$,
air viscosity $\mu = 10^{-5}~\textnormal{Pa}\,\textnormal{s}$
and $\textnormal{Re} = 1$ yields $U_0 = 10^{-4}~ \textnormal{m}\,\textnormal{s}^{-1}$. Such a velocity would only require
$\Delta P_0 \approx 10^{-9}~\textnormal{Pa}$ (estimated assuming Poiseuille flow in a tube).
Using a laboratory sand with $\perm \approx 10^{-12}~\textnormal{m}^2$ yields $S \approx 10$,
$\textnormal{Da} \approx 10^{-10}$, $\alpha \in [1, 10]$.
\end{example}

\begin{example}[Cerebrospinal fluid flow in sub-arachnoid space and brain cortex]\label{ex:csf}
The brain cortex can be considered a porous medium with $\perm \approx 10^{-18} - 10^{-16} ~\textnormal{m}^2$~\cite{Holter2017}.
The brain is surrounded by the sub-arachnoid space (SAS), a shallow void layer ($L_0 \approx 2~\textnormal{mm}$)
filled with a water-like fluid ($\mu = 10^{-3}~\textnormal{Pa}\,\textnormal{s}$). Typical Stokes velocities in SAS
range between $10^{-3}$ and $1~\textnormal{cm}\,\textnormal{s}^{-1}$, and typical pressure gradients are on the order of
$1~\textnormal{Pa}$ which gives $S \in [5\cdot 10^{-4}, 5 \cdot 10^{-1}]$,
$\textnormal{Da} \in [2.5\cdot10^{-13}, 2.5\cdot10^{-11}]$, $\alpha \in [1, 10]$.  
\end{example}

\subsection{Robustness study}\label{sec:param_sweep}
\crefrange{ex:channel}{ex:csf} reveal that $S \in [10^{-5}, 10^{1}]$, $\text{Da} \in [10^{-14}, 10^{0}]$,
$\alpha \in [0, 10^{2}]$ cover a wide range of relevant applications.
Following the problem and solver setup described in \cref{ex:naive},
we report iterations of the preconditioned MinRes solver using the
three Stokes-Darcy formulations \eqref{eq:ds_weak}, \eqref{eq:dsLM_weak} and \eqref{eq:ds_robin_weak} with the numerically exact (LU-inverted) preconditioners \eqref{eq:daq_precond}, \eqref{eq:dsLM_precond}
and \eqref{eq:ds_robin_precond}. Discretization in terms of both FEM and FVM
is considered. We recall that due to the experimental setup, in particular, the unit sized
scaling parameters, cf. \ref{eq:mms}, the ranges identified in \cref{sec:params} are effectively the
ranges for $\mu$, $\perm$ and $\alpha$.

\subsubsection{Preconditioning $\mathcal{A}^\text{Tr}$}
Using discretization by FEM, we investigate formulation 
\eqref{eq:ds_weak} with preconditioner \eqref{eq:daq_precond}. Both conforming
$\vec{P}_2$-${P}_1$-${P}_2$ and non-conforming $\vec{C}\vec{R}_1$-${P}_0$-${P}_0$
elements are used. For the latter, we employ a facet stabilization
\cite{burman05} (see also \eqref{eq:CR_stab}). We refer to \cref{sec:mms} for
approximation properties of these schemes for the Stokes-Darcy problem.

Starting with the conforming $\vec{P}_2$-${P}_1$-${P}_2$ elements,
\cref{fig:fem_iters_TH_CR} summarizes performance of \eqref{eq:daq_precond}
for the Trace formulation. Specifically, in each subplot
corresponding to a fixed value of $\mu$ (varies in row), we plot the iteration
count for different refinement levels, six different values of $k$ indicated by color and
four different values of the slip coefficient $\alpha$. It can be seen that the iterations
are bounded in mesh size as well as the material parameters. Specifically,
between $24$ and $53$ iterations are required for convergence in all cases.
Furthermore, the (bounded) condition numbers of the preconditioned
systems are reported in \cref{sec:meet_dirichlet}. 
\begin{figure}
  \centering
  \includegraphics[width=\textwidth]{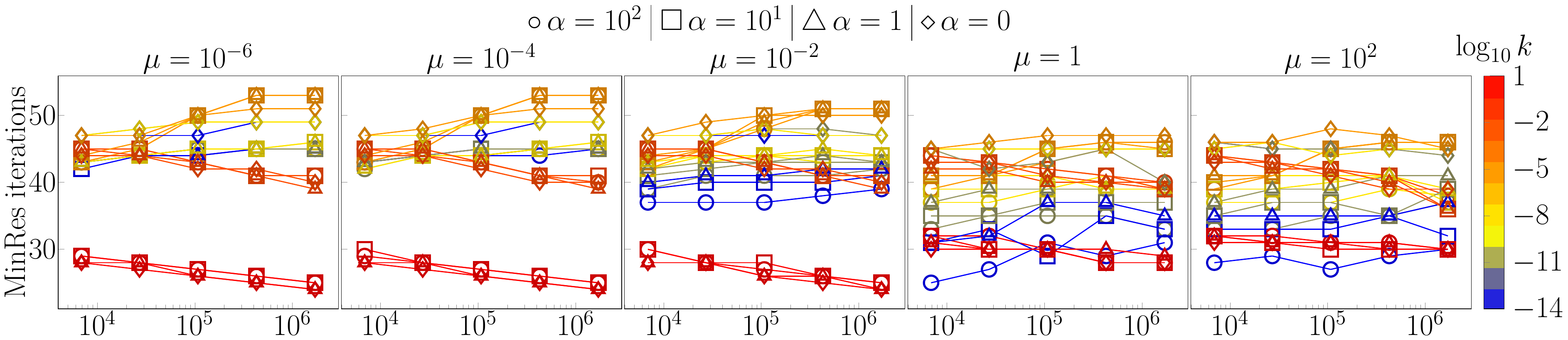}
  \includegraphics[width=\textwidth]{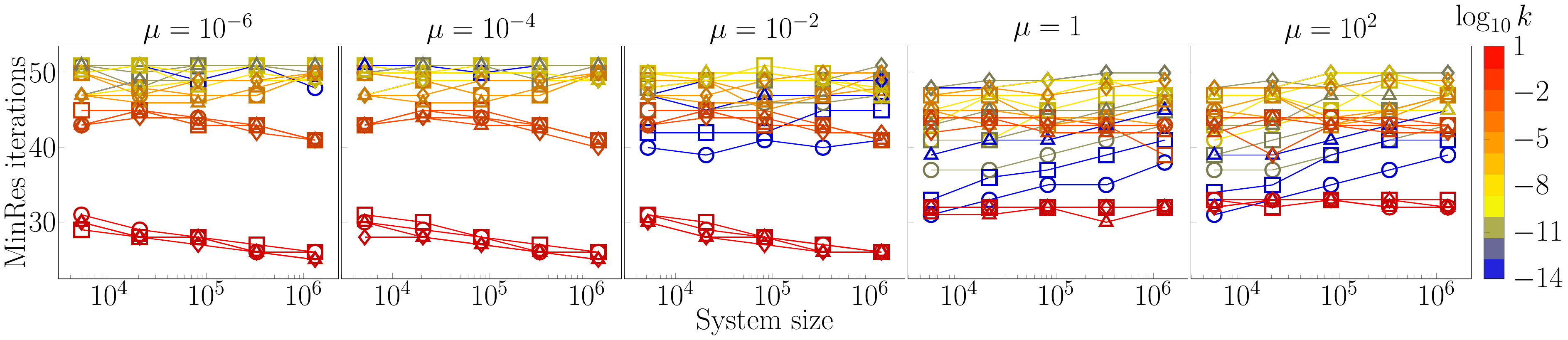}  
  \vspace{-25pt}  
  \caption{
    Performance of preconditioner \eqref{eq:daq_precond} for the trace
    formulation \eqref{eq:ds_weak} and the model problem from \cref{ex:naive} with parameter ranges
    identified in \cref{sec:params}.
    Discretization by $\vec{P}_2$-${P}_1$-${P}_2$ elements (top) and
    $\vec{C}\vec{R}_1$-${P}_0$-${P}_0$ elements (bottom) 
    using stabilization \cite{burman05} (see also \eqref{eq:CR_stab}).
  }
   \label{fig:fem_iters_TH_CR}
\end{figure}

For non-conforming $\vec{C}\vec{R}_1$-${P}_0$-${P}_0$ elements, the results
are given in the bottom panel of \cref{fig:fem_iters_TH_CR}. We observe that the iterations appear bounded, varying betweeen $25$ and $51$.
However, there is a modest increase\footnote{
  Between the smallest and the largest system considered the iterations grow by 10 while the system size increases 
  by 3 orders of magnitude.
} with $h$ for $\perm=10^{-14}$ and $\mu\geq 1$.
We attribute this growth to round-off errors when inverting the preconditioner since
the pressure block is then scaled with $10^{-16}$-$10^{-14}$.

\subsubsection{Preconditioning $\mathcal{A}^\text{La}$ and $\mathcal{A}^\text{Ro}$}\label{sec:la_ro_robustness}
We discuss robustness of \eqref{eq:dsLM_precond} and \eqref{eq:ds_robin_precond} for
the multiplier formulation \eqref{eq:dsLM_weak} and the Robin formulation
\eqref{eq:ds_robin_weak}. Linear solver (MinRes) iterations over a large range
of parameters are shown in \cref{fig:fvm_iters_LM_Robin} and confirm
parameter-robustness in both cases, with iteration counts between $10$ and $39$ for \eqref{eq:dsLM_precond}-preconditioned $\mathcal{A}^\text{La}$ and iteration counts between $1$ and $48$ for \eqref{eq:ds_robin_precond}-preconditioned $\mathcal{A}^\text{Ro}$. We note that in particular when the ratio $\kovermu = \mu^{-1}k$
is small, the reported iteration counts are very small (even one in the most extreme case) but stable for varying system sizes.
We can attribute this to the specific configuration of the test case.
With $\kovermu \ll 1$
the contribution $\betanormal \tracenormalprime \tracenormal $ in operator \eqref{eq:fvm_op} dominates the Stokes block.
We recall that $\betanormal := \betanormalfull$. However, the right-hand side of the linear system only scales with $\mu$ for our particular case and both normal velocity and normal velocity gradient are zero in the exact solution~\eqref{eq:mms}-\eqref{eq:mms_rhs}. In this setting, the linear solver manages to reduce the very large initial defect (due the combination of random initial guess in the range $\left[0,1\right)$, large operator norm, and small right hand side) by the requested factor of $10^8$ in only one iteration. We remark that in this case the approximation of the solution is rather poor and a stricter convergence criterion would be required to obtain an accurate solution. However, this does not diminish the observation that the iterations are bounded. In consistency with all other results, we therefore report the results for the specified reduction of $10^8$. This particularity does not affect the multiplier formulation since the term $\betanormal \tracenormalprime\tracenormal $ is not present in operator \eqref{eq:dsLM_op}.
  To fully convince the reader, we additionally report condition numbers of the
  discrete preconditioned operators 
  in \cref{sec:fvm_cond_lm_and_robin}.
The results show that the condition number stays between $5$ and $17$ for all reported parameter combinations.
We note that the condition number estimates involving $\mathcal{B}^{\text{Ro}}$ are reported over
a smaller range of mesh sizes than in  \cref{fig:fvm_iters_LM_Robin} since the computations
require an expensive assembly of an inverse of sum of two inverted matrices.

%
\begin{figure}
  \centering
  \includegraphics[width=\textwidth]{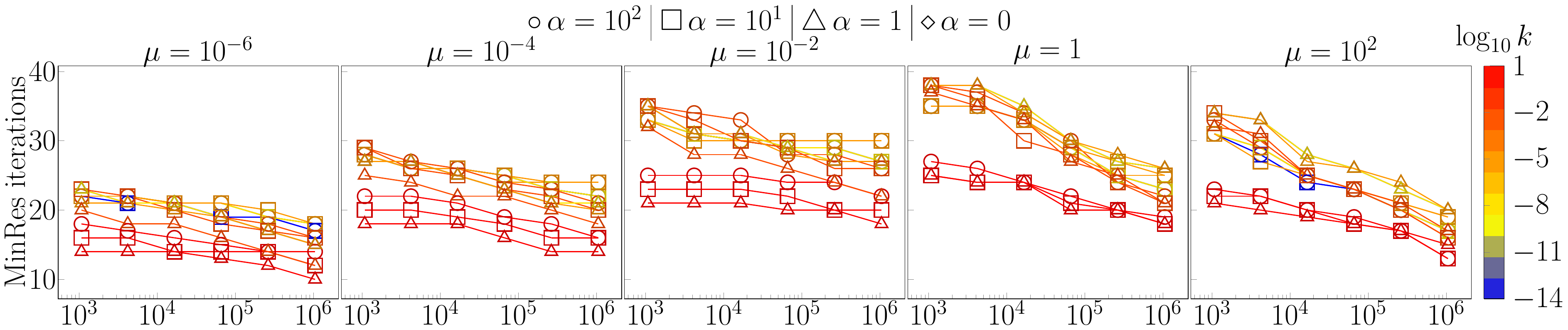}
  \includegraphics[width=\textwidth]{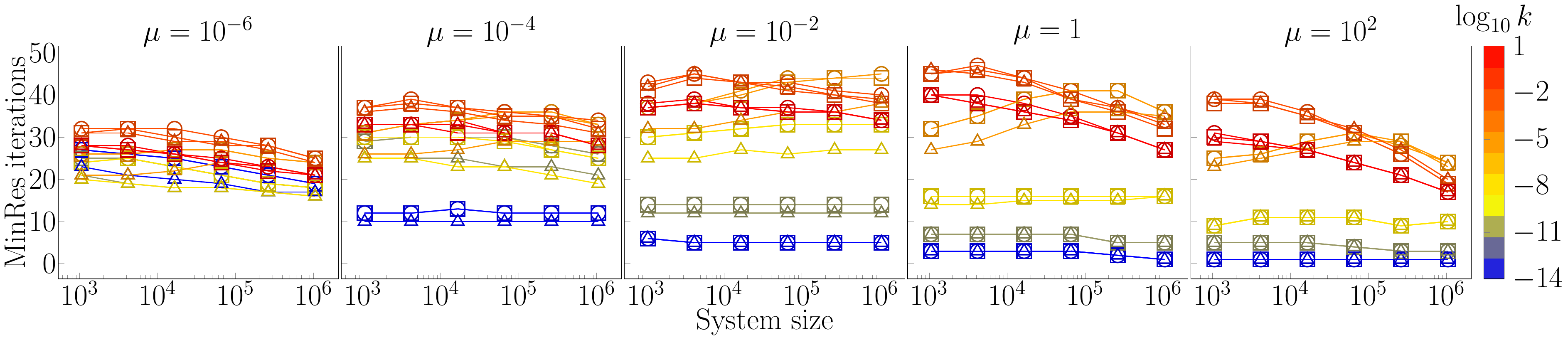}
  \vspace{-25pt}  
  \caption{
    Iteration counts for the \eqref{eq:dsLM_precond}-preconditioned multiplier formulation \eqref{eq:dsLM_weak}
    (top) and \eqref{eq:ds_robin_precond}-preconditioned Robin formulation \eqref{eq:ds_robin_weak} (bottom). 
    Discretization with FVM as described in \cref{sec:app:fvm}.
  }
  \label{fig:fvm_iters_LM_Robin}
\end{figure}

Preconditioners $\mathcal{B}^{\text{La}}$ and $\mathcal{B}^{\text{Ro}}$ 
were also investigated using non-conforming $\vec{C}\vec{R}_1$-${P}_0$-${P}_0$(-$P_0$) elements.
Results collected in \cref{fig:fem_iters_CR_LM_Robin} confirm
robustness of both preconditioners. The number of iterations remained between
$24$ and $50$ for \eqref{eq:dsLM_precond}-preconditioned $\mathcal{A}^\text{La}$ and
between $16$ and $50$ for $\mathcal{A}^\text{Ro}$ with preconditioner \eqref{eq:ds_robin_precond}.
\begin{figure}
  \centering
  \includegraphics[width=\textwidth]{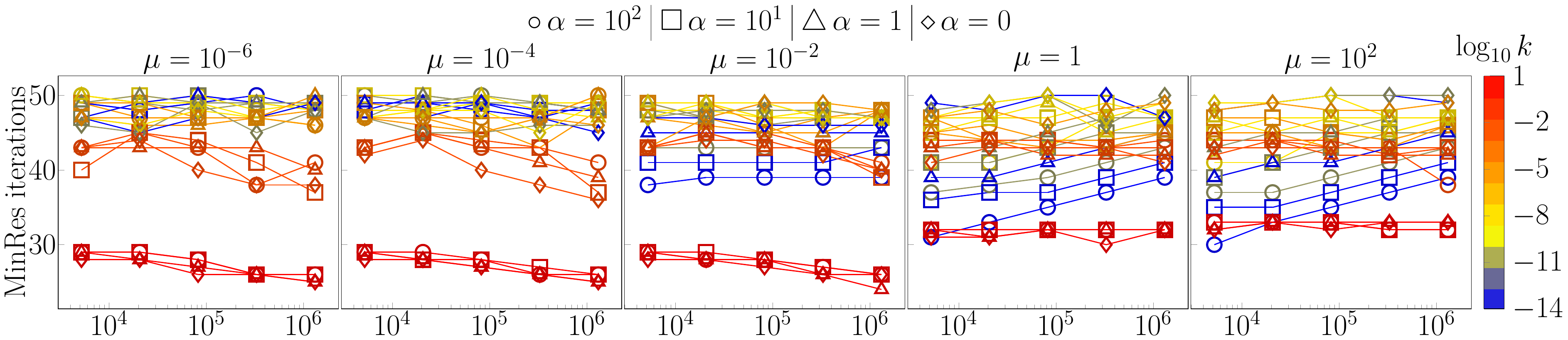}
  \includegraphics[width=\textwidth]{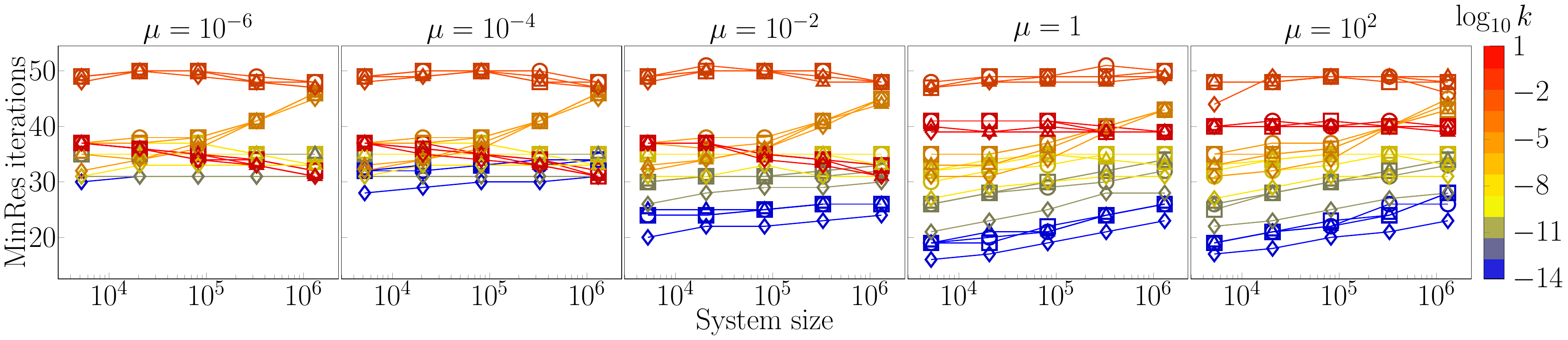}
  \vspace{-25pt}  
  \caption{
    Iteration counts for the \eqref{eq:dsLM_precond}-preconditioned Lagrange multiplier formulation \eqref{eq:dsLM_weak}
    (top) and Robin formulation \eqref{eq:ds_robin_weak} with preconditioner \eqref{eq:ds_robin_precond} (bottom). 
    Discretization by $\vec{C}\vec{R}_1$-${P}_0$-${P}_0$(-${P}_0$) elements. 
  }
  \label{fig:fem_iters_CR_LM_Robin}
\end{figure}

\subsection{Three-dimensional examples}\label{sec:efficiency}
The robustness study of \cref{sec:param_sweep} concerned a two-dimensional
setup leading to a rather small interface with only a few hundred cells in $\Gamma_h$.
Moreover, the preconditioners $\mathcal{B}^{\text{Tr}}$, $\mathcal{B}^{\text{La}}$ and $\mathcal{B}^{\text{Ro}}$
were always computed exactly. To address the efficiency of the preconditioners
in more practical scenarios, we next apply the proposed Stokes-Darcy solvers to two three-dimensional model problems.
In particular, we investigate the effect of approximating the action of the
preconditioner blocks in terms of off-the-shelf multilevel methods. In addition,
the problems are chosen such that we go beyond the assumptions on 
the interface (not a closed surface) and the boundary conditions (interface intersects with Neumann boundary) introduced at the end of \cref{sec:problem} to simplify the theoretical analysis.

In the following two examples,
we investigate solvers for two of the proposed Stokes-Darcy formulations:
(A) the Trace formulation \eqref{eq:ds_weak} with preconditioner \eqref{eq:daq_precond}
and FEM discretization and (B)
the Lagrange multiplier formulation \eqref{eq:dsLM_weak} with preconditioner
\eqref{eq:dsLM_precond} discretized by FVM.
We remind the reader that the FEM and FVM implementations differ in the software
stack. In particular, all FEM results using direct solvers are obtained with MUMPS \cite{MUMPS}, while FVM results use UMFPACK \cite{davis2004a}.
The FEM results rely on Hypre's BoomerAMG \cite{hypre}, while for FVM the algebraic multigrid of dune-istl~\cite{Blatt2010AMG} is used. Moreover, the results are computed with different hardware setup (A) 
Ubuntu workstation with AMD Ryzen Threadripper 3970X 32-Core processor and
128GB of memory, (B) openSUSE workstation
with AMD Ryzen Threadripper 3990X 64-Core processor and 270GB of memory.
However, in both cases the computations
are run in serial restricted to one CPU\footnote{Single threaded execution of all solver components is enforced by setting \texttt{OPM\_NUM\_THREADS=1}.}.
Finally (and going more beyond the presented theory), the Stokes block in the FVM
operators \eqref{eq:dsLM_op} and \eqref{eq:dsLM_precond} is not symmetric due to
a non-symmetric stencil in the current implementation of boundary condition \eqref{eq:BJS} for the case of
reentrant corners in the Stokes domain. However, the asymmetry is localized to the few degrees of freedom
associated with the interface.
As symmetry is a strict requirement for MinRes, we present GMRes iterations instead.

To evaluate efficiency of the proposed preconditioners, we compare their numerically exact
realization to approximations in terms of multilevel methods. For (A) and BoomerAMG, the different
approximations correspond to computing the action of each block by increasing
numbers (same for each block for simplicity) of $\text{V}(2, 2)$ AMG cycles per application of
the preconditioner. We used default settings except for the aggregation threshold which is set to $0.7$, the recommended value for $3d$ problems. For (B) and Dune::AMG, the number of smoother iterations $n$ on
each level of a $\text{V}(n,n)$-cycle was varied.

In addition, we compare the solvers,
with the analogues of the na{\"i}ve precondiner presented in \cref{ex:naive},
that is, $\mathcal{B}^{\text{Tr}}$, respectively $\mathcal{B}^{\text{La}}$ with the
fractional operator omitted. (Moreover, in this case the pressure block of the preconditioner reads $-\kovermu(\Delta+I)$
to avoid the singularity due to the Neumann boundary conditions on
$\partial\Omega_D\setminus\Gamma$.)

\subsubsection{Channel flow over porous hill}\label{sec:channel}
\begin{figure}[htb]
  \centering
  \includegraphics[width=0.82\textwidth]{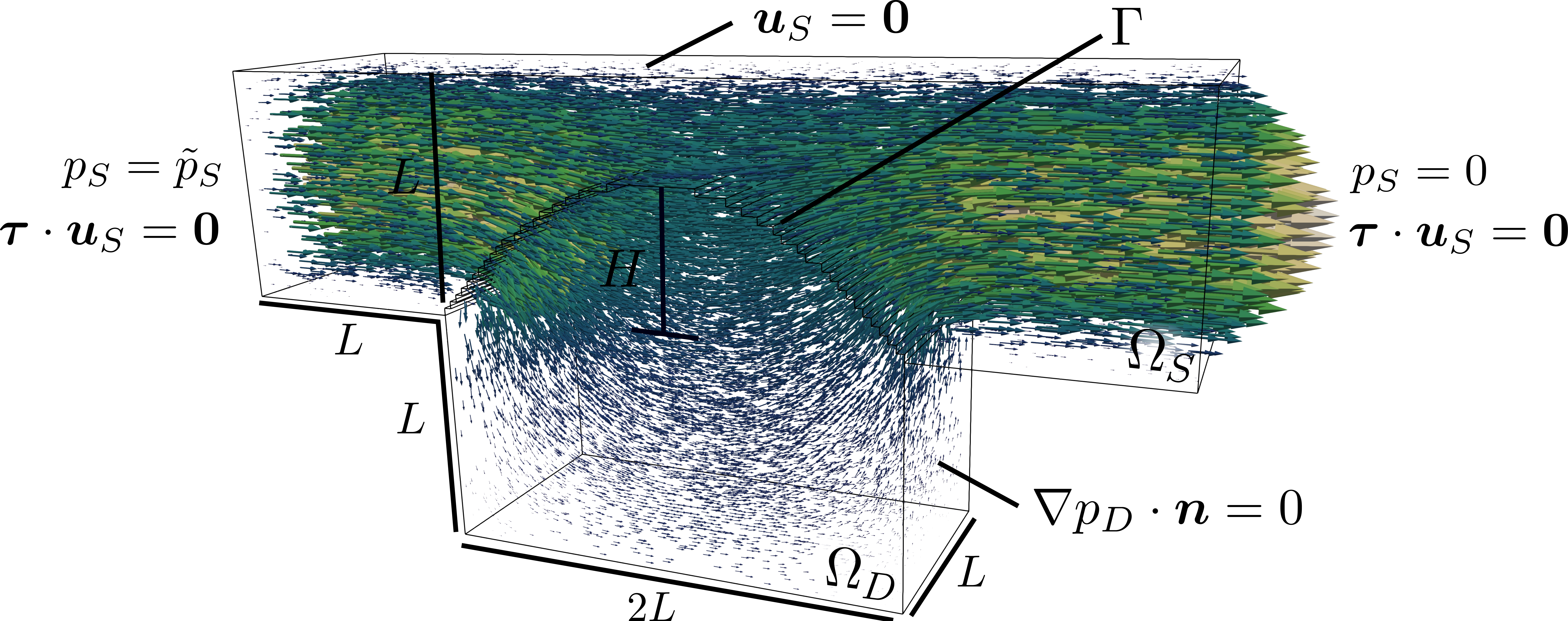}
  \vspace{-10pt}
  \caption{
    Setup for channel flow example in \cref{sec:efficiency}. All boundaries are no-flow boundaries
    except for the right and left sides of the channel where we prescribe pressure and zero
    tangential velocities with $\tilde{p}_S = 10^{-8}~$. The dimensions are given by
    $L = 0.5$, $H = 0.3$ and the interface $\Gamma$ is chosen as a section of a cylinder
    with radius $R = H/2 + L^2/(2H)$. Arrows visualize the resulting velocity field for
    $\mu = 10^{-3}$, $\perm = 10^{-2}$, $\alpha = 1$, where longer, lighter arrows correspond to higher velocities.
  }
  \label{fig:3d_sandpit}
\end{figure}
\begin{table}
\begin{minipage}{\textwidth}
  \centering
  \caption{
    Exact and approximate preconditioners for Stokes-Darcy problem in \cref{fig:3d_sandpit}
    using the Trace formulation \eqref{eq:ds_weak} and preconditioner \eqref{eq:daq_precond}.
    Discretization by $\vec{P}_2$-${P}_1$-${P}_2$. MinRes iterations until
    convergence (reducing preconditioned residual norm by factor $10^{8}$)
    with the preconditioner computed exactly (bLU) and by $j\text{V}(2, 2)$
    cycles of AMG are shown.
    The numbers in parenthesis represent
    aggregate solver setup and solver run time rounded to full seconds. (Top) $k=10^{-2}$,
    (bottom) $k=10^{-5}$. In addition timings for solving $\mathcal{A}^{\text{Tr}}$ using 
    a \emph{direct} solver and preconditioned MinRes solver with the exact \emph{na{\"i}ve}-preconditioner of \cref{ex:naive} are listed.
  }
  \vspace{-10pt}
  \label{tab:sandpit_fem}
  \footnotesize{
    
    \begin{tabular}{rr|||l|lll||ll}
      \hline
      dofs & $\lvert V_h \rvert$ & bLU & 1V(2,2) & 2V(2,2) & 4V(2,2) & direct\textsuperscript{a} & na{\"i}ve\\
      \hline
3562 & 107    & 84 (1) & 92 (1) & 85 (1) & 84 (1)           & - (1)     & 98 (1) \\
13452 & 293   & 89 (2) & 103 (2) & 91 (3) & 89 (5)          & - (1)    & 106 (2) \\
69554 & 1023  & 88 (11) & 112 (20) & 92 (30) & 88 (54)      & - (4)   & 106 (9)  \\
468646 & 3671 & 88 (173) & 129 (265) & 101 (388) & 89 (646) & - (98) & 108 (138) \\
      \hline
3562 & 107       & 88 (1) & 108 (1) & 98 (1) & 95 (1)          & - (1) & 1065 (2)    \\
13452 & 293      & 92 (2) & 123 (2) & 106 (4) & 101 (6)        & - (1) & 1538 (14)   \\
69554 & 1023    & 93 (11) & 143 (23) & 110 (33) & 98 (55)      & - (4) & 1659 (110)  \\
468646 & 3671  & 97 (180) & 164 (314) & 122 (439) & 105 (716)  & - (98) & 1661 (1293)\\
      \hline
  \end{tabular}\\
  \textsuperscript{a} MUMPS 
  }
\end{minipage}
\end{table}
\begin{table}[htb]
\begin{minipage}{\textwidth}
  \caption{
    Exact and approximate preconditioners for Stokes-Darcy problem in \cref{fig:3d_sandpit}
    using formulation \eqref{eq:dsLM_weak} and preconditioner $\mathcal{B}^\text{La}$~\eqref{eq:dsLM_precond}. Discretization with FVM (Staggered-TPFA). GMRes iterations (reducing preconditioned residual norm by factor $10^{8}$) are shown. In parenthesis, we provide wall clock times (aggregate solver setup and runtime) rounded to full seconds. Tables shows results for $\perm = 10^{-2}$ (top), $\perm = 10^{-5}$ (middle), $\perm = 10^{-12}$ (bottom). In addition timings for solving $\mathcal{A}^{\text{La}}$ using 
    a \emph{direct} solver and preconditioned GMRes solver with the exact \emph{na{\"i}ve}-preconditioner constructed by omission of the fractional component in $\mathcal{B}^{\text{La}}$
    are included.
    }
    \label{tab:sandpit_fvm}
    \centering
    \vspace{-10pt}
    \footnotesize{
    \begin{tabular}{r r ||| l|lll||ll}\hline
     dofs & $|V_h|$   &  bLU &  1V(1,1) & 1V(2,2) & 1V(4,4) & direct\textsuperscript{a} & na{\"i}ve\\\hline
   8640 &   208 &   60 (1)&   83 (1)&   71 (1)&   64 (1) &    - (1) &   89 (1)\\ 
  66342 &   832 &   64 (11)&  109 (7)&   92 (7)&   81 (7) &    - (6) &   86 (13)\\ 
 517552 &  3264 &   66 (337)&  156 (132)&  128 (134)&  109 (146) &    - (468) &   84 (374)\\ 
   \hline
   8640 &   208 &   69 (1)&  111 (1)&   96 (1)&   89 (1) &    - (1) &  330 (3)\\ 
  66342 &   832 &   82 (13)&  145 (8)&  127 (8)&  112 (9) &    - (6) &  403 (53)\\ 
 517552 &  3264 &   91 (411)&  201 (161)&  169 (163)&  147 (180) &    - (485) &  449 (1431)\\ 
   \hline
   8640 &   208 &   69 (1)&  113 (1)&   97 (1)&   89 (1)&   - (1) & 4411 (158) \\ 
  66342 &   832 &   83 (14)&  147 (11)&  130 (11)&  116 (13)&  - (6) & n/c\textsuperscript{b} \\
  517552 &  3264 &   99 (437)&  203 (235)&  180 (256)&  161 (312)&    - (533) & n/c\textsuperscript{b} \\
  \hline
   \end{tabular}\\
   \textsuperscript{a} UMFPACK \quad \textsuperscript{b} not converged in under $10'000$ iterations
}
\end{minipage}
\end{table}

As the first model problem,
we consider viscous flow over
a porous medium with a curved interface, see \cref{fig:3d_sandpit}. Let $\alpha=1$, $\mu=10^{-3}$ and $k\in\{10^{-2}, 10^{-5}\}$.
The fluid motion is driven by a pressure difference between the inlet and 
outlet where the (non-standard) boundary conditions $p_S=\tilde{p}_{S}$ (inlet, $p_S=0$ on the outlet)
and $\utangent\cdot\vels=\vec{0}$ (see e.g. \cite{girault1990curl}) are prescribed.

On the rest of the fluid
domain, we enforce $\vels = \vec{0}$, while the boundary of the porous domain is
impermeable (homogeneous Neumann boundary conditions).
Therefore, newly, the interface intersects (mixed boundaries) $\Gamma^{\vec{u}}_S$ and
$\Gamma^{u}_D$. The fact that $\Gamma$ is incident to the Dirichlet boundary on the 
Stokes side translates to a modification of the preconditioner
such that the fractional operator is now constructed with Dirichlet boundary conditions,
see \cref{sec:meet_dirichlet} for further details.

Performance of the $\mathcal{B}^{\text{Tr}}$-preconditioned formulation
\eqref{eq:ds_weak} discretized with the $\vec{P}_2$-${P}_1$-${P}_2$ FEM is summarized in \cref{tab:sandpit_fem}. It can
be seen that exact preconditioners lead to iterations bounded in refinement
with little sensitivity to the change in permeability. In addition, the LU-based
preconditioners are noticeably faster\footnote{
  Due to the used (mostly default) settings the timings of AMG should be considered a pessimistic bound for the performance.
} than the AMG-based approximation. 
We remark that with LU at most 30\% of the reported time was spent in the setup
phase which was dominated by factorization of the blocks.
 To give an example of the cost of the eigensolver, for the finest interface mesh
  reported in \cref{tab:porousblocks_fem}, $\lvert V_h \rvert=13976$, assembly 
  of the fractional block takes
  $256~\text{s}$. 
  However, the presence of the resulting (large) dense block in the matrix of the pressure
preconditioner also affects factorization time and the cost per Krylov iteration.

For preconditioners
realized by AMG cycles robustness in $h$ requires at least four V cycles if
$k=10^{-2}$ while 8 cycles are needed for $k=10^{-5}$. This result supports our
observation (not reported here) that black-box algebraic multigrid is not
a parameter-robust preconditioner for the pressure block in  \eqref{eq:daq_precond}.
Specifically, AMG struggles when the interface term dominates the Laplacian
in $\Omega_D$. Finally, in agreement with \cref{ex:naive} for $k=10^{-5}$, the na{\"i}ve preconditioner
leads to considerably more iterations (and slower run time) than $\mathcal{B}^{\text{Tr}}$.
However, none of the iterative approaches outperform the direct solver for the reported system sizes.

Performance of the $\mathcal{B}^{\text{La}}$-preconditioned formulation
\eqref{eq:dsLM_weak} discretized with the Staggered-TPFA FVM is summarized in \cref{tab:sandpit_fvm}.
In comparison with the FEM results but in consistency with observations in $2d$ examples in \cref{sec:param_sweep},
the solvers based on FVM and exact preconditioner initially show a slight increase of the number of iterations
with refinement (in particular for small $k$). However, we point out that the difference in the number of iterations between consecutive grid refinements gets smaller and smaller (similar to what can be seen for the condition numbers in \cref{sec:fvm_cond_lm_and_robin}).
While the solver with exact $\mathcal{B}^\text{La}$ appears parameter-robust, it is evident that the na{\"i}ve preconditioner (missing the fractional component) is not robust in $k$.
When approximating all blocks with AMG, the fastest execution times could be achieved. In comparison with BoomerAMG, Dune::AMG uses a faster but less accurate interpolation strategy, which leads to considerably faster execution time per iteration. Increasing the number of smoother iterations reduces iteration counts but due to the increased cost per iteration does not result in a better performance. Moreover, the Dune::AMG-based solver does not show robustness with grid refinement, even for a large number of smoother iterations (we tested up to $64$). However, the Dune::AMG-based solver appears robust in the model parameters. 

\subsubsection{Embedded porous blocks}%
\begin{figure}[htb]
  \centering
  \includegraphics[width=0.76\textwidth]{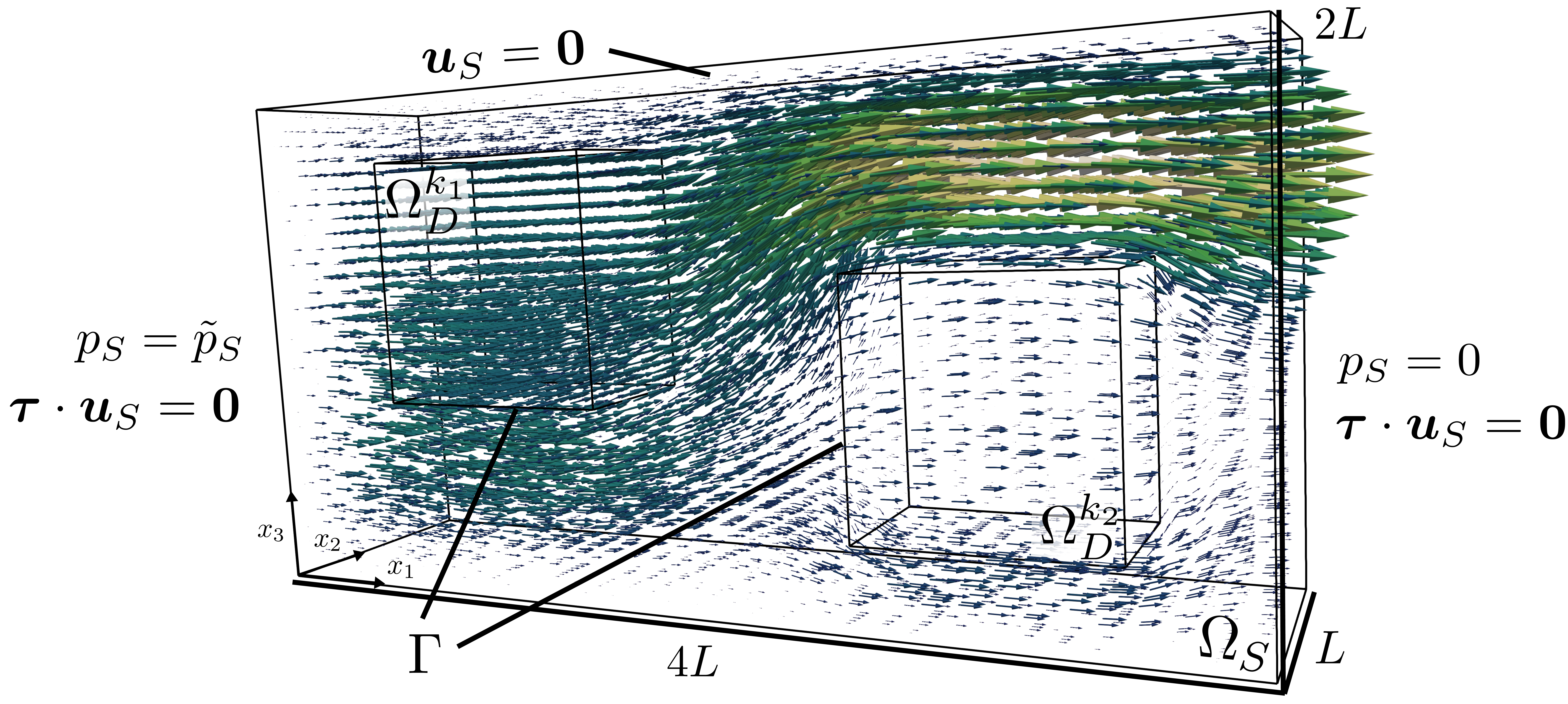}
  \vspace{-10pt}
  \caption{Setup for porous blocks example in \cref{sec:efficiency}. All boundaries are no-flow boundaries
    except for the right and left sides of the channel where we prescribe pressure and zero
    tangential velocities with $\tilde{p}_S = 10^{-8}~$. The channel dimension are given by
    $L = 0.5$ and the blocks are $\Omega_D^{k_1} = [L/2, 3L/2]\times[2L/8, 7L/8]\times[3L/4, 7L/4]$ and $\Omega_D^{k_2} = [5L/2, 7L/2]\times[L/8, 6L/8]\times[1L/4, 5L/4]$ and are assigned different permeabilities $k_1 = 10^{-1}$, $k_2 = 10^{-3}$. Moreover, $\mu = 10^{-3}$, $\alpha = 1$. The coarsest FVM discretization is a structured (anisotropic) rectangular cuboid mesh with $16\times16\times16$ cells. FEM results are computed with unstructured tetrahedral meshes.
  }
  \label{fig:sandbanks}
\end{figure}
\begin{table}
  \centering
  \caption{
    Exact and approximate preconditioners for Stokes-Darcy problem in \cref{fig:sandbanks}
    using the Trace formulation \eqref{eq:ds_weak} and preconditioner \eqref{eq:daq_precond}.
    Discretization by $\vec{P}_2$-${P}_1$-${P}_2$. Legend as in \cref{tab:sandpit_fem}
  }
  \vspace{-10pt}  
  \label{tab:porousblocks_fem}
  \footnotesize{
    \begin{tabular}{r @{  }r|||l|lll||ll}
      \hline
      dofs & $\lvert V_h \rvert$ &bLU & 1V(2,2) & 2V(2,2) & 4V(2,2) & direct & na{\"i}ve\\
      \hline
7256 & 484      & 123 (1) & 143 (2) & 129 (3) & 125 (4)             & - (1)     & 240 (2)\\
25260 & 1212    & 125 (4) & 152 (8) & 131 (13) & 127 (21)           & - (1)     & 257 (6)\\
124732 & 3768   & 125 (37) & 162 (84) & 133 (127) & 126 (221)       & - (12)   & 264 (46)\\
836293 & 13976  & 126 (932) & 189 (1426) & 144 (1909) & 129 (3049)  & - (370) & 275 (721)\\
 \hline
  \end{tabular}
  }
\end{table}
\begin{table}
  \centering
  \caption{
    Exact and approximate preconditioners for Stokes-Darcy problem in \cref{fig:sandbanks}
    using the Lagrange multiplier formulation \eqref{eq:dsLM_weak} and preconditioner \eqref{eq:dsLM_precond}.
    Discretization by FVM. Legend as in \cref{tab:sandpit_fvm}.
  }
  \vspace{-10pt}  
  \label{tab:porousblocks_fvm}
  \footnotesize{
    \begin{tabular}{r @{  }r|||l|ll||ll}
\hline
 dofs & $\lvert V_h \rvert$ &bLU & 1V(1,1) & 1V(4,4)  & direct & na{\"i}ve \\ \hline
  16144 &   608 &  107 (4) &  151 (3)&  118 (3)&   - (1) & 257 (6) \\ 
 122432 &  2432 &  109 (85) &  192 (62)&  147 (70)&   - (64) & 258 (138) \\ 
 952576 &  9728 &  109 (3861) &  281 (3458)&  195 (3310)&   - (5617) & 187 (4806.8) \\ \hline
  \end{tabular}
  }
\end{table}

In the second and final example, we consider viscous channel flow
past and through two porous inclusions with different permeabilities, see \cref{fig:sandbanks}.
From the point of view of assumptions of \cref{sec:problem}, the novel feature is the fact
that the interface is now formed by two closed surfaces.

Iterations counts and runtime estimates for various solvers are shown in \cref{tab:porousblocks_fem} (FEM) and
\cref{tab:porousblocks_fvm} (FVM). In general, the conclusions from \cref{sec:channel}
apply to the new example as well. In particular, exact preconditioners
$\mathcal{B}^{\text{Tr}}$, $\mathcal{B}^{\text{Tr}}$ yield iteration counts that are stable in mesh size.

\section{Conclusions and outlook}\label{sec:summary}
Our work concerned monolithic preconditioning of symmetric formulations of
the coupled Stokes-primal Darcy problem which were motivated by differences in handling the
interface coupling that are natural to finite element and finite volume methods. Parameter robust
preconditioners for each of the three formulations were constructed based on the well-posedness of the
problems established within a unifying functional framework.
The proposed preconditioners are based on norms in fractional Sobolev spaces. Using discretization
in terms of both FEM and FVM our numerical results
demonstrated the parameter-robustness in several examples partly going beyond the presented theory in
terms of boundary conditions and interface configuration. However, efficiency of the proposed solvers
is currently sub-optimal
due to the realization of the pressure preconditioner,
in particular, the reliance on the spectral form of the fractional interface operators.

To improve efficiency of the proposed preconditioners scalable techniques for the
parameter-robust approximation of the components, in particular, the pressure block, will be
addressed in the future work. To possibly improve the efficiency further,
the use of lower/upper-triangular preconditioners or approximations of full Schur complement factorizations
could be investigated. Here, a reduction of the iteration count is expected but the cost-benefit ratio
of such an approach for the presented cases remains to be seen. Finally, extensions
of the proposed preconditioners to more complex physics such as the Navier-Stokes-Darcy problem
may be addressed in future work.

\bibliographystyle{siamplain}
\bibliography{references}

\newpage
\appendix


\section{Non-conforming discretizations}\label{sec:discrete_ops}
This section provides additional details on discretization of the Stokes-Darcy
operators by lowest-order non-conforming FEM (i.e. $\vec{C}\vec{R}_1$ elements for the space 
$\bV_{S, h}$ and $\Pnot$ elements for $Q_{D, h}$) and FVM. In the following
we denote as $\mathcal{F}^{i}(\Omega_h)$ the set of interior facets
of a given mesh $\Omega_h$ of generic bounded Lipschitz domain $\Omega$.

\subsection*{Non-conforming finite element discretization} To
obtain stable discretization of the Stokes subproblem in \eqref{eq:ds_weak}-\eqref{eq:ds_robin_weak}
on the space $\bV_{S, h}=\bV_{S, h}(\Omega_h)$ constructed in terms of Crouzeix-Raviart $\vec{C}\vec{R}_1$ element
we employ the facet stabilization \cite{burman05}. That is, the operator $-\nabla\cdot(2\mu\symgradop)$
is discretized as
\begin{equation}\label{eq:CR_stab}
\langle-\nabla\cdot(2\mu\symgradop(\bu)), \bv\rangle := (2\mu\symgrad{\bu}, \symgrad{\bv})_{\Omega}
+ \sum_{F\in \mathcal{F}^{i}({\Omega_{h}})} (\frac{2\mu}{\lvert F \rvert}\jump{\bu}, \jump{\bv})_F, \quad\forall \bu, \bv\in \bV_{h},
\end{equation}
where $\jump{\vec{u}}=\vec{u}|_{K^{+}}-\vec{u}|_{K^{-}}$
in which $K^{\pm}$ denote the two cells sharing the facet $F$. Moreover, we recall that
$\mu$ is assumed to be constant and that $h_K$ measures the distance between
facet midpoint and centroids/circumcenters of the connected cells.

Approximation of the Laplace operator in the space of piece-wise constant
functions $Q_h=Q_h(\Omega_h)$ uses a two-point flux approximation, that is,
we let
\[
\langle -\Delta p, q \rangle := 
\sum_{F\in \mathcal{F}^{i}({\Omega_{h}})} (\frac{1}{2\avg{h_K} }\jump{p}, \jump{q})_F + (h_K^{-1}p, q)_{\Gamma^p}, \quad\forall p, q \in Q_{D},
\]
where $\jump{{p}}={p}|_{K^{+}}-{p}|_{K^{-}}$, $\avg{p}=\tfrac{1}{2}(p|_{K^{+}}+p|_{K^{-}})$, 
and $\Gamma^p\subseteq \partial\Omega$ is the part of the domain boundary
with Dirichlet data. This definition is also used when assembling the
fractional operator via the eigenvalue problems in \eqref{eq:eigv} and in 
\cref{sec:meet_dirichlet}.
%
\subsection*{Finite volume discretization}\label{sec:app:fvm}
The herein employed finite volume discretization method (FVM) is a combination of a staggered face-centered finite volume scheme
(Staggered) for the Stokes momentum balance equation \eqref{eq:stokes_mom} and a cell-centered finite volume scheme with two-point flux approximation (TPFA) for the Stokes mass balance equation \eqref{eq:stokes_mass}, the Darcy equation \eqref{eq:darcy} and, if using the Lagrange multiplier formulation, eigenvalue problem \eqref{eq:eigv}. For the description of the Staggered FVM, we refer to~\cite{Harlow1965,Shiue2018,Schneider2020}. Here, due to the immediate relevance for the construction of the Lagrange multiplier and Robin formulations \eqref{eq:dsLM_weak} and \eqref{eq:ds_robin_weak} as well as the assembly of eigenvalue problem \eqref{eq:eigv}, we briefly review the cell-centered TPFA FVM.

Let us consider as example the Darcy equation \eqref{eq:darcy}.
We integrate \eqref{eq:darcy} over each control volume $\cell \in \Omega_{D,h}$, apply the divergence theorem
and approximate the interface fluxes by a discrete numerical flux approximation $\Psi_{K,F}$ for each face $F$ of $K$, 
\begin{equation*}
   - \int_{\partial\cell} \mu^{-1}k \nabla p_D \cdot \vec{n}_{\cell,F} \,\textrm{d}s = \int_{\cell} f_D \,\textrm{d}x \quad \Rightarrow \quad \sum\limits_{F \in \partial K} \Psi_{K_D,F} = \meas{\cell} f_D(\vec{x}_\cell),
\end{equation*}
where $\vec{n}_{\cell,F}$ is a unit normal vector on $F$ pointing out of $K$ and $\vec{x}_\cell$ is the centroid of $\cell$.
A two-point flux approximation for $\Psi_{K_D,F}$ on inner facets is then given by
\begin{equation}\label{eq:TPFA}
\Psi_{\cell,\face} := \frac{1}{\mu}\frac{t_{\cell,\face}t_{L,\face}}{t_{\cell,\face} + t_{L,\face}} (p_\cell - p_L), \quad t_{\cell,\face}
:= k_K\frac{\vec{d}_{{\cell},\face} \cdot \vec{n}_{\cell,\face}}{\vert\vert\vec{d}_{\cell,\face} \vert\vert^2}\meas{\face},
\end{equation}
where $p_\cell$ denotes the average cell pressure in cell $\cell$, $k_K$ the permeability of cell $K$, $\vec{d}_{\cell,\face}$  the vector connecting $\vec{x}_\cell$ and an integration point on $F$ (e.g. centroid), and $L$ is a neighboring cell sharing $F$ with $K$. Note that \eqref{eq:TPFA} also applies for surface grids, for instance, for the approximation of eigenvalue problem \eqref{eq:eigv} on curved surfaces. On the boundary, we either specify $\Psi_{\cell,\face}$ directly (Neumann boundary conditions) or compute $\Psi_{\cell,\face} = \mu^{-1}k_K t_{\cell,\face} (p_\cell - p_{\partial\cell})$ where $p_{\partial\cell}$ is the given boundary data on $F \subset \partial\Omega_D$. (That is, Dirichlet data, or interface pressure $p_{\Gamma}$ in the case of the coupling interface.)
We remark that approximation \eqref{eq:TPFA} is only consistent on $\vec{K}$-orthogonal grids~\cite{Aavatsmark1998}.

\section{Numerical tests and manufactured solution}\label{sec:mms}
For the numerical grid convergence tests and parameter-robustness tests,
we work with the manufactured solution given in \cite{Shiue2018} for unit parameters
$\mu = 1$, $k = 1$, $\alpha = 1$ as
\begin{subequations}\label{eq:mms}
\begin{align}
   \vels &= \begin{bmatrix}
      -\frac{1}{\pi} \exp(x_2) \sin(\pi x_1) \\
      (\exp(x_2) - \exp(1)) \cos(\pi x_1)
   \end{bmatrix} &\text{in} \quad \Omega_S,\\
   p_S &= 2 \exp(x_2) \cos(\pi x_1), &\text{in} \quad \Omega_S,\\
   p_D &= (\exp(x_2) - x_2\exp(1)) \cos(\pi x_1) &\text{in} \quad \Omega_D, \\
\intertext{
where $\Omega_D = [0,1]\times[0,1]$, $\Omega_F = [0,1]\times[1,2]$.
Moreover, for the formulation with Lagrange multiplier,
}
    \lambda &= p_D(x_1, x_2 = 1) = 0 &\text{on} \quad \Gamma,
\end{align}
\end{subequations}
where $\Gamma = [0,1]\times \{1\}$.
To obtain the same solution over the whole range of parameters, we
use the following source terms
\begin{subequations}\label{eq:mms_rhs}
\begin{align}
   \bfS &:= \begin{bmatrix}
      \frac{1}{\pi} \exp(x_2) \sin(\pi x_1) (\mu - \mu\pi^2 - 2\pi^2)  \\
      \cos(\pi x_1) (\mu \left((\pi^2 + 1)\exp(x_2) - \pi^2\exp(1))\right) + 2(1-\mu)\exp(x_2))
    \end{bmatrix},\\
    f_D &:= \frac{k}{\mu} \cos(\pi x_1) \left((\pi^2 + 1)\exp(x_2) - \pi^2 x_2 \exp(1)\right),
\end{align}
\end{subequations}
and modified coupling conditions
\begin{subequations}\label{eq:mms_coupling}
\begin{align}
  \utangent\cdot\stressvar\cdot \uonormal + \betatangent \utangent\cdot\vels &= h^\Gamma_\tau,
  &&h^\Gamma_\tau := \left(\betatangent - \mu \right) \frac{1}{\pi} \exp(x_2) \sin(\pi x_1),\\
  \uonormal \cdot\stressvar\cdot \uonormal + {p}_D &= h^\Gamma_n,
  &&h^\Gamma_n :=  2(\mu - 1) \exp(1)\cos(\pi x_1),\\
  \vels\cdot \uonormal + \kovermu\nabla p_D\cdot \uonormal &= g^\Gamma,
  &&g^\Gamma := 0.
\end{align}
\end{subequations}
The functions $h^\Gamma_\tau$, $h^\Gamma_n$, and $g^\Gamma$ ensure
that the conditions are satisfied independent of the choice of parameters.
Note that the choice of data in \eqref{eq:mms_coupling} only modifies the
right-hand side while the problem operators remain unchanged.

This setting enables code verification in terms of grid convergence tests in all parameter settings.
For grid convergence tests, the errors for the finite element schemes are reported in $L^2$ and $H^1$
norms. Using $\vec{P}_2$-$P_1$-$P_2$ elements for \eqref{eq:ds_weak} quadratic
convergence in all the variables in their respective norms is expected.
Discretization by $\vec{C}\vec{R}_1$-$P_0$-$P_0$(-$P_0$) in all the formulations yields
a first order scheme.

The errors for the finite volume scheme are computed in the following discrete $L^2$ norm
\begin{equation}\label{eq:discrete-fvm-norm}
\lVert u \rVert_\text{FV} := \left( \sum\limits_{K \in {\Omega}_{h}} |K| u_{K}^2 \right)^{\frac{1}{2}}.
\end{equation}
It is well known that with the typical flux reconstruction schemes, based on a two-point
flux approximation on structured Cartesian grids, second order super-convergence at cell
centers (pressures) and face centers (Stokes velocity components)
is obtained \cite{Li2014,Droniou2017,Schneider2018}.

We report error convergence of the FEM schemes for all the formulations
in \cref{fig:fem_cvrg}. Expected (or faster) convergence is observed
in all cases. We remark that the observed quadratic convergence of the interfacial pressure
$p|_{\Gamma}$ in \eqref{eq:dsLM_weak} is likely due to the zero exact solution
in the manufactured setup. Error convergence for the FVM schemes is reported for
formulations \eqref{eq:dsLM_weak} and \eqref{eq:ds_robin_weak} in \cref{fig:fvm_cvrg}.
Quadratic convergence in the discrete norm \eqref{eq:discrete-fvm-norm} is observed for all the variables.
\begin{figure}
  \centering
  \includegraphics[height=0.29\textwidth]{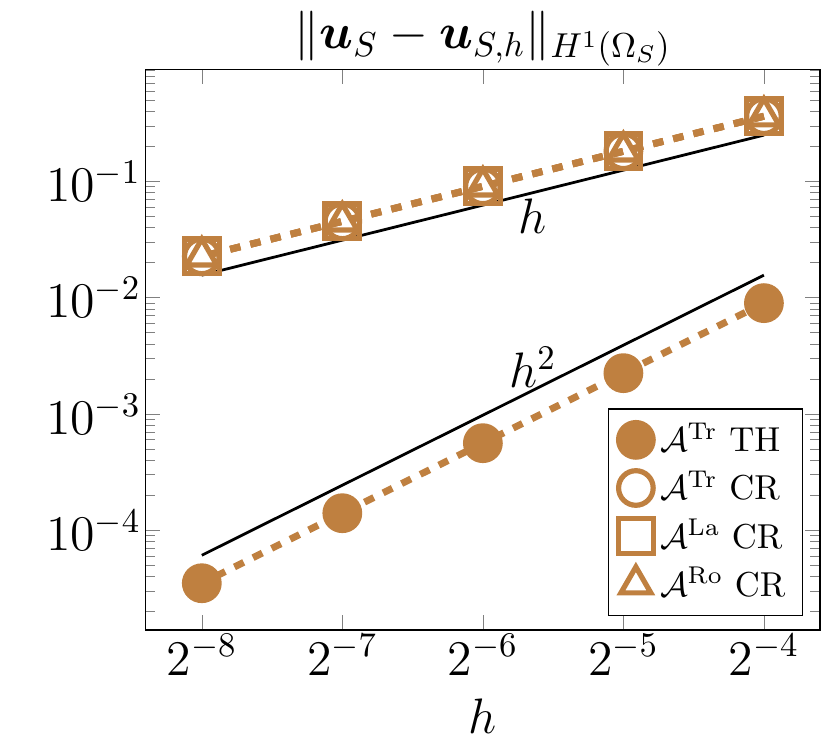}
  \includegraphics[height=0.29\textwidth]{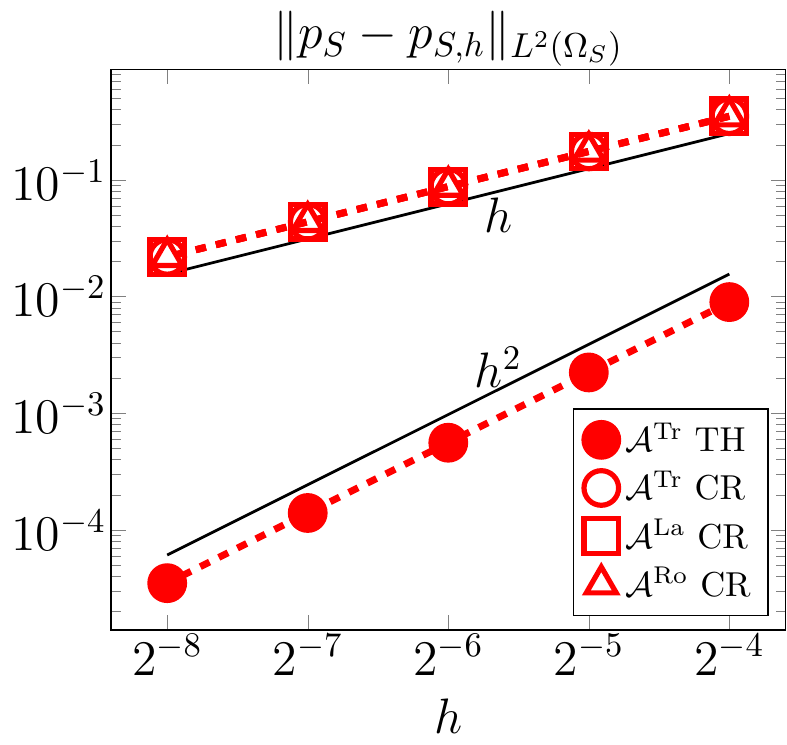}
  \includegraphics[height=0.29\textwidth]{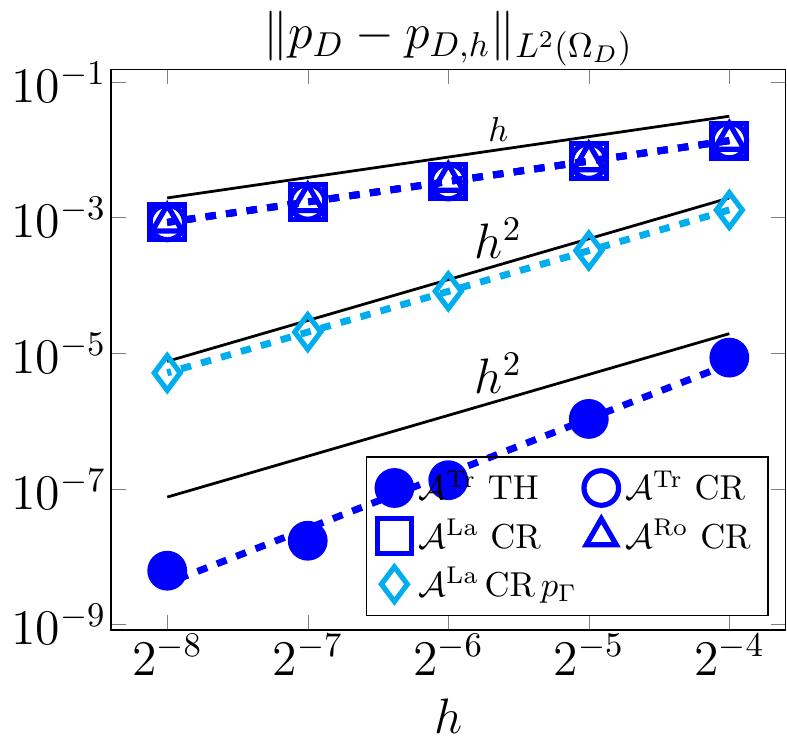}
  \vspace{-12pt}
  \caption{Approximation properties of FEM discretizations for the manufactured problem \eqref{eq:mms}
    with unit parameters using the Trace formulation \eqref{eq:ds_weak}, multiplier formulation \eqref{eq:dsLM_weak}
    and Robin formulation \eqref{eq:ds_robin_weak}. Only \eqref{eq:ds_weak}
    is considered with $\vec{P}_2$-$P_1$-$P_2$ while $\vec{C}\vec{R}_1$-$P_0$-$P_0$(-$P_0$)
    is used for all formulations. $L^2(\Gamma)$-error of interface pressure in \eqref{eq:dsLM_weak} is plotted in cyan color.
  }
  \label{fig:fem_cvrg}
\end{figure}
\begin{figure}
  \centering
  \includegraphics[height=0.29\textwidth]{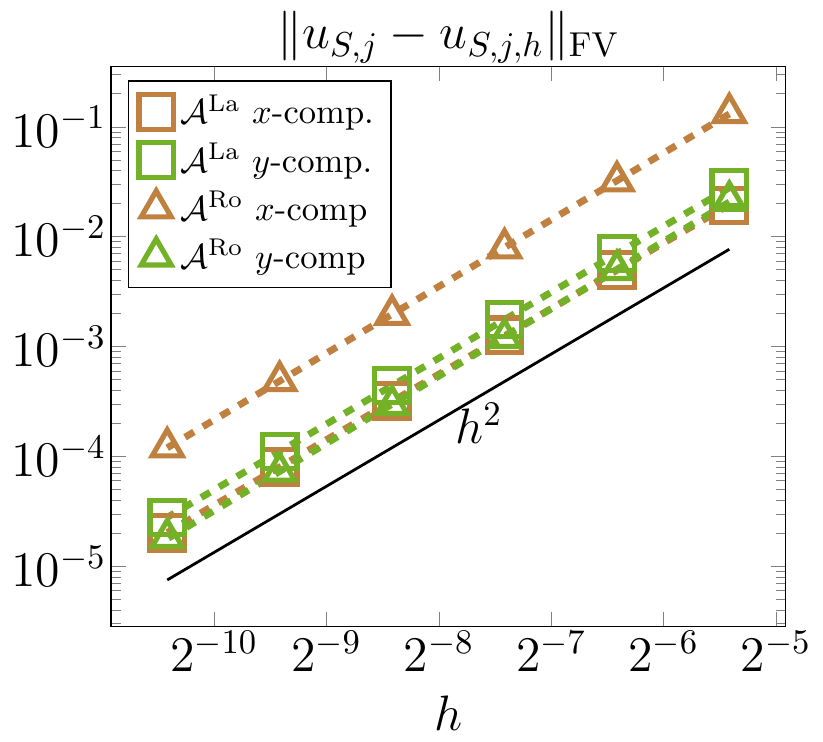}
  \includegraphics[height=0.29\textwidth]{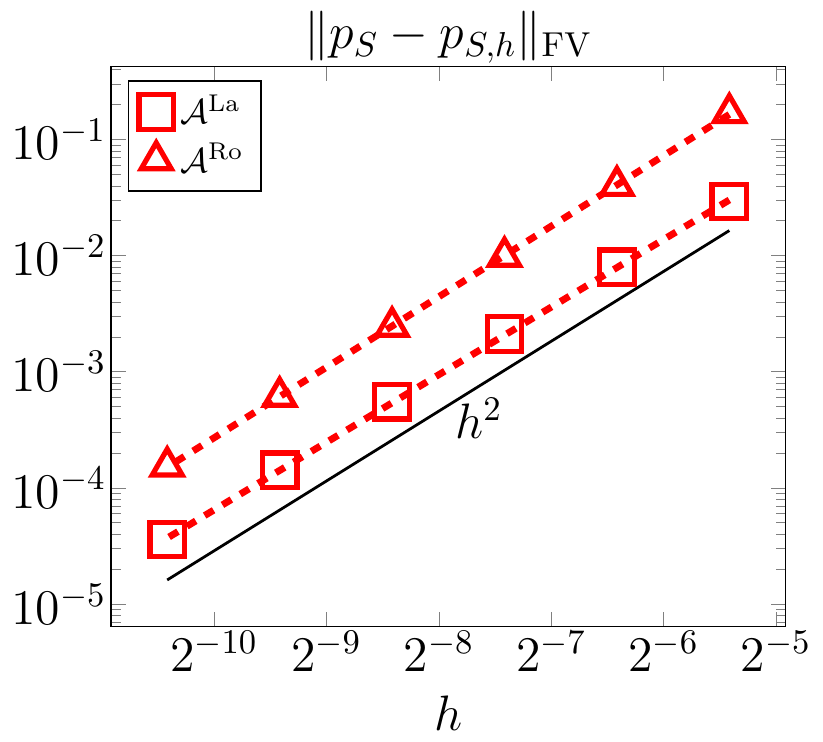}
  \includegraphics[height=0.29\textwidth]{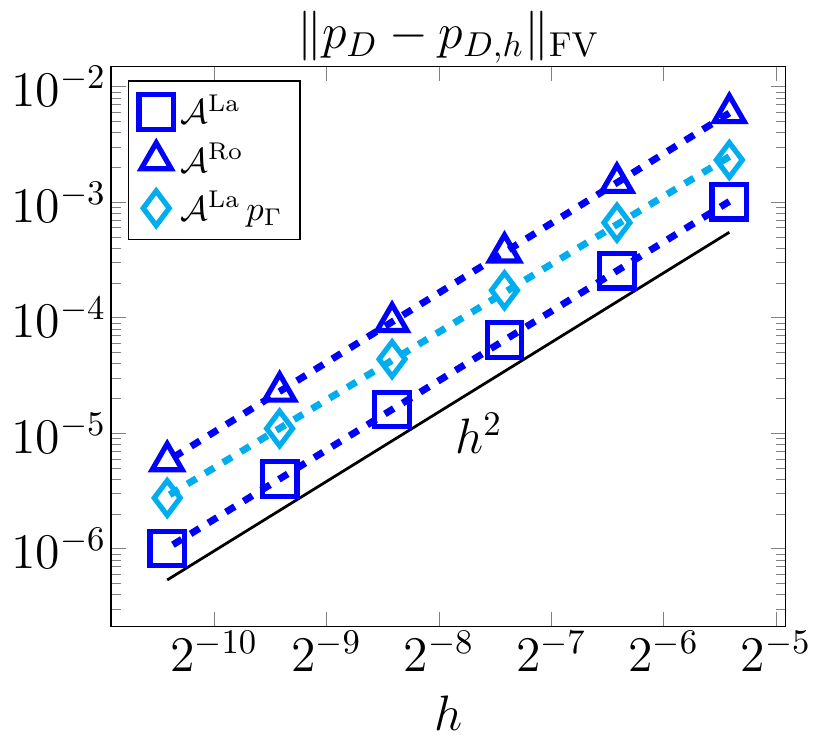}
  \vspace{-12pt}
  \caption{Approximation properties of FVM (Staggered-TPFA) discretization for the manufactured problem \eqref{eq:mms}
    with unit parameters using the Lagrange multiplier formulation \eqref{eq:dsLM_weak}
    and the Robin formulation \eqref{eq:ds_robin_weak}. $u_{S,x}$ and $u_{S,y}$ denote the components
    of $\vels$ (the degrees of freedom for the respective components have different locations and control volumes in the staggered FVM).
    For \eqref{eq:dsLM_weak} the error of interface pressure error measured in $L^2$-norm \eqref{eq:discrete-fvm-norm} on $\Gamma$
    is plotted in cyan.
    }
  \label{fig:fvm_cvrg}
\end{figure}


\section{Interface intersecting Dirichlet boundaries}\label{sec:meet_dirichlet}
The analysis of \cref{sec:abstract} and robustness study of \cref{sec:param_sweep} 
assume that the interface
$\Gamma$ intersects the Neumann boundaries of both the Darcy and the Stokes domain.
This fact was reflected by the $\| \cdot \|_{-\frac12, \Gamma}$ norm used in the analysis and
in the preconditioner construction through eigenvalue problem \eqref{eq:eigv}.
If instead, $\Gamma$ intersects the Dirichlet boundaries of the respective problems
\eqref{eq:daq_precond} no longer defines a parameter-robust preconditioner\footnote{
  Visual inspection of the spectrum in this case reveals that the number of eigenvalues
  unbounded in $\perm$ corresponds to the number of degrees of freedom of the
  intermediate trace space $V_h$ (see \cref{sec:discrete_precond}) associated
  with $\partial\Gamma_h$. Since this number is finite in a two-dimensional problem 
  the issue typically does not affect performance of iterative solvers. However,
  this is not the case for $d=3$ as then the number of unbounded modes increases
  with $\lvert \partial\Gamma_h \rvert$.
  } for formulation \eqref{eq:ds_weak}, see \cref{tab:eigs_daq_dir_no00}.

\begin{table}
  \centering
  \caption{
    Condition numbers of \eqref{eq:daq_precond}-preconditioned $\mathcal{A}^{\text{Tr}}$
    when $\Gamma$ intersects
    Dirichlet boundaries of both subproblems. Geometry of \cref{ex:naive}
    is used with $\mu=1$, $\alpha=1$. Parameter sensitivity is due to incorrect,
    namely, $\lVert \cdot\rVert_{-\frac12, \Gamma}$, control at the interface.
  }
  \label{tab:eigs_daq_dir_no00}
  \vspace{-10pt}  
  \footnotesize{
    \begin{tabular}{c|ccccc}
      \hline
      \backslashbox{$\perm$}{$h$} & $2^{-2}$ & $2^{-3}$ & $2^{-4}$ & $2^{-5}$ & $2^{-6}$\\
      \hline
      $1$	&7.37	&7.46	&7.47	&7.46	&7.45 \\
$10^{-1}$	&9.16	&9.26	&9.27	&9.26	&9.26 \\
$10^{-2}$	&18.21	&18.52	&18.58	&18.59	&18.58\\
$10^{-4}$	&30.59	&34.94	&37.84	&39.13	&39.51\\
      \hline
    \end{tabular}
  }
\end{table}

However, the theory of \cref{sec:abstract} and resulting preconditioners
can be extended to more general cases. In particular,
the fact that $\Gamma$ intersects with Dirichlet boundaries  
translates into a modification of the interface norm to be used
in the preconditioner, that is, the pressure on the interface shall be controlled in
$\| \cdot \|_{H^{-\frac12}_{00}(\Gamma)}$. We recall that $H^{\frac12}_{00}(\Gamma)$ is a subspace
of $H^{\frac12}(\partial\Omega_D)$ containing functions that vanish on $\partial\Omega_D\setminus\Gamma$,
see \cite{galvis2007non} for more details.
Hence, eigenvalue problem \eqref{eq:eigv} used in the construction of the discrete preconditioner
is replaced by $-\Delta_{\Gamma} u=\lambda u$ on $\Gamma$ \emph{and} $u=0$ on $\partial\Gamma$, i.e.
Dirichlet conditions are enforced.\footnote{As with Neumann boundaries in \eqref{eq:eigv}, the actual boundary data is irrelevant since it does not modify the operator.}
Finally, we note that above we have set $\mu=1$ for simplicity. In general case, the parameter scaling is analogous to \eqref{eq:eigv}.

Using $\vec{P}_2$-$P_1$-$P_2$ elements, \cref{fig:fem_eigs_TH} reports
spectral condition numbers of the Stokes-Darcy Trace formulation
\eqref{eq:ds_weak} with preconditioner \eqref{eq:daq_precond}. The
geometry is taken from \cref{ex:naive}, however for the Dirichlet case, the placement of Dirichlet
and Neumann boundaries is interchanged: Neumann boundaries on top and bottom edges; Dirichlet boundary conditions on the lateral edges which intersect with $\Gamma$. Parameter ranges from \cref{sec:params} are used. We observe stable condition
numbers $C$ in the range $7.3 \leq C \leq 18.5$ (interface meeting Dirichlet boundary)
and $6.2 \leq C \leq 16.5$ (interface meeting Neumann boundaries, i.e. the case analyzed
in \cref{sec:abstract} and numerically investigated in \cref{sec:numeric}).
\begin{figure}[]
  \centering
  \includegraphics[width=\textwidth]{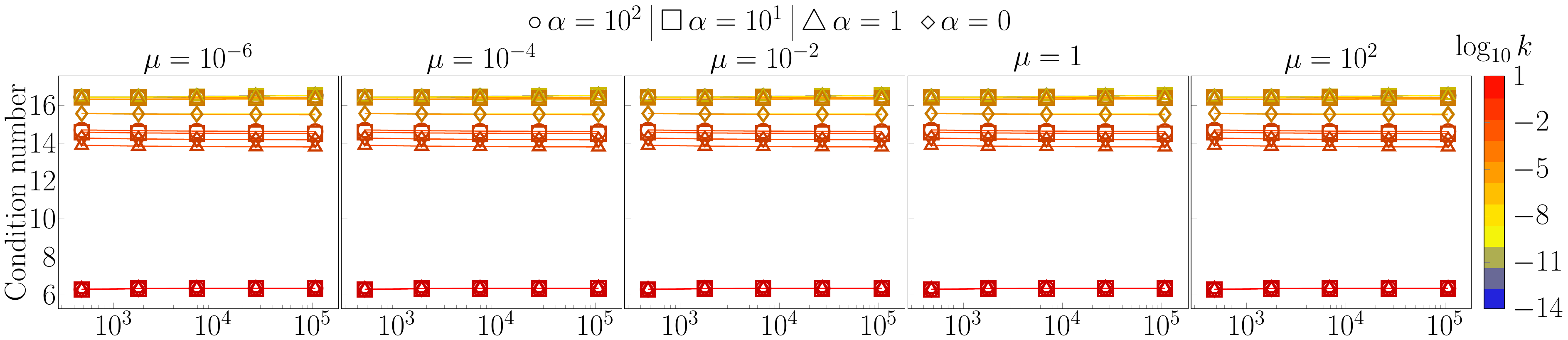}
  \includegraphics[width=\textwidth]{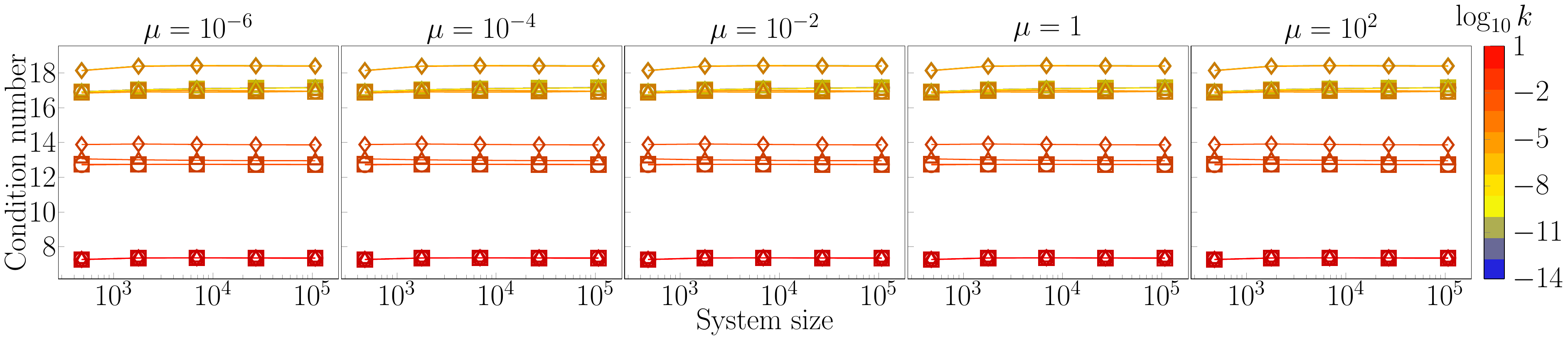}  
  \vspace{-25pt}
  \caption{
    Condition numbers for \eqref{eq:daq_precond}-preconditioned
    formulation \eqref{eq:ds_weak} across the parameter ranges from \cref{sec:params}, with
    interface $\Gamma$ intersecting Neumann boundaries (top) or Dirichlet
    boundaries (bottom). The case of Dirichlet boundaries requires modification of the
    preconditioner as described in \cref{sec:meet_dirichlet}. Problem \eqref{eq:ds_weak} is assembled on geometry
    defined in \cref{ex:naive} and discretized by $\vec{P}_2$-$P_1$-$P_2$ elements.
    }
  \label{fig:fem_eigs_TH}
\end{figure}


\section{FVM condition numbers for $\mathcal{B}^\text{La}$ and $\mathcal{B}^\text{Ro}$}\label{sec:fvm_cond_lm_and_robin}

We report in \cref{fig:fvm_cond_nolm} the condition numbers corresponding to the numerical tests of \cref{sec:param_sweep} and the MinRes iteration results reported in \cref{sec:la_ro_robustness}.
\begin{figure}
  \centering
  \includegraphics[width=\textwidth]{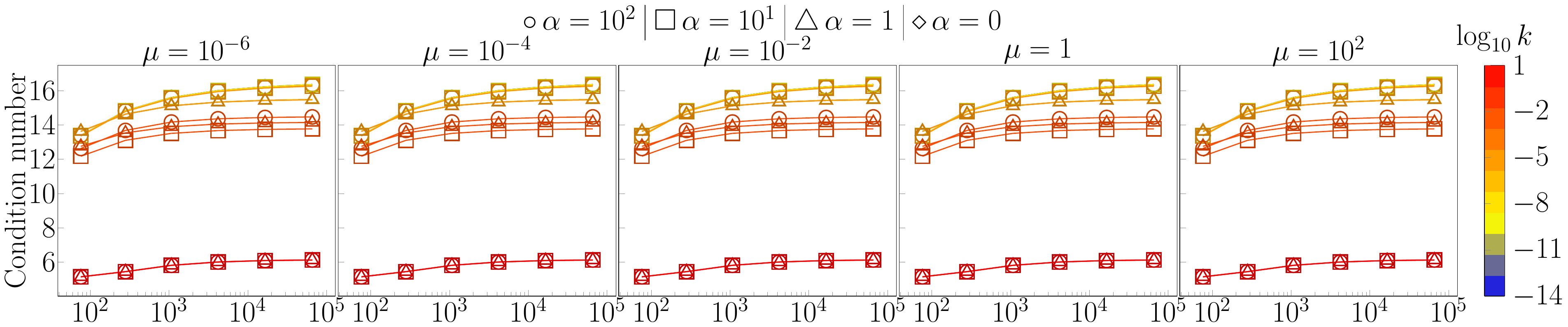}
  \includegraphics[width=\textwidth]{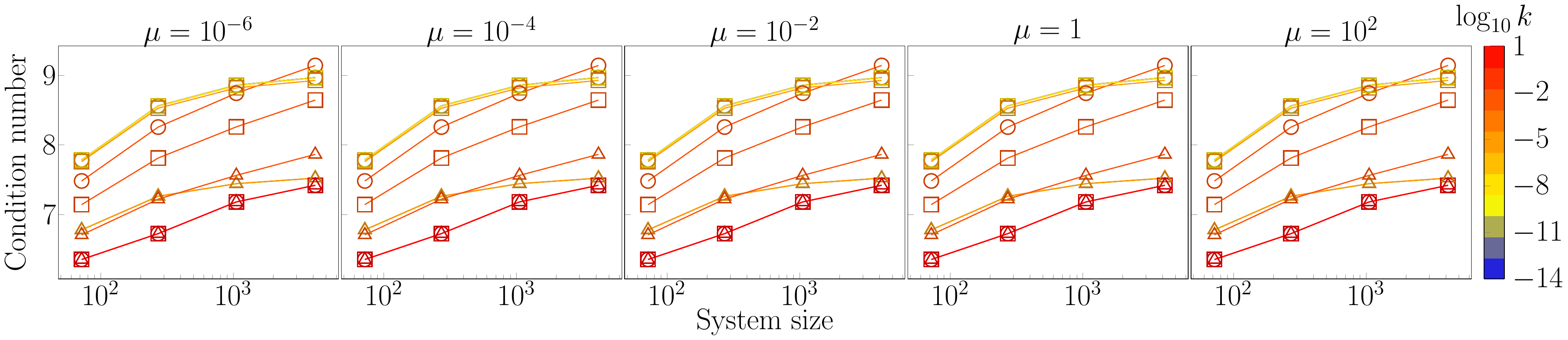}
  \vspace{-25pt}
  \caption{
    Condition numbers for the \eqref{eq:dsLM_precond}-preconditioned multiplier
    formulation \eqref{eq:dsLM_weak} and the \eqref{eq:ds_robin_precond}-preconditioned FVM
    formulation \eqref{eq:ds_robin_weak} (bottom) across the parameter ranges from \cref{sec:params}. Both formulations are discretized with FVM~(\cref{sec:app:fvm}). Not all $k$ are visible since the data overlaps with larger values of $k$.}
  \label{fig:fvm_cond_nolm}
\end{figure}

\end{document}